\documentclass[12pt, letterpaper, oneside]{memoir} 

\usepackage[main=english]{babel}
\usepackage[T1]{fontenc}
\usepackage[utf8]{inputenc}

\usepackage{amsmath, amssymb, amsthm}
\usepackage{tikz, tikz-cd}
\usepackage{mathtools, etoolbox}
\usepackage{graphicx}
\usepackage{enumitem}
\usepackage{aliascnt}
\usepackage[hidelinks ,pdfencoding=auto]{hyperref}
\usepackage[autostyle, english=american]{csquotes}

\usepackage{setspace,hyperref,caption,array,color}
\usepackage{amsmath,amssymb,amsthm,mathtools,bm}
\usepackage{multicol,tabularx,tikz,tikz-cd,sseq}
\usetikzlibrary{decorations.shapes}

\usepackage[backend=biber, style=alphabetic, giveninits=true]{biblatex}

\usepackage{bm} 

\allowdisplaybreaks

\usepackage[double, 1.25inmargin]{uothesis-memoir}

\usepackage{indentfirst} 

\usepackage{lipsum} 

\DeclarePairedDelimiterX\Set[1]{\lbrace}{\rbrace}%
 {  #1 }

\DeclarePairedDelimiterX\abs[1]{\lvert}{\rvert}{
\ifblank{#1}{\:\cdot\:}{#1}}
\DeclarePairedDelimiterX\norm[1]{\lVert}{\rVert}{
\ifblank{#1}{\:\cdot\:}{#1}}
\DeclarePairedDelimiterX\innerp[2]{\langle}{\rangle}{#1,#2}

\newcommand{\mynewtheorem}[2]{%
  \newaliascnt{#1}{thmdummy}
  \newtheorem{#1}[#1]{#2}
  \aliascntresetthe{#1}
  \expandafter\def\csname #1autorefname\endcsname{#2}
}

\theoremstyle{plain}
    \mynewtheorem{theorem}{Theorem}
    \mynewtheorem{lemma}{Lemma}
    \mynewtheorem{corollary}{Corollary}
    \mynewtheorem{proposition}{Proposition}
    \newtheorem*{theorem*}{Theorem}
    \newtheorem*{proposition*}{Proposition}
    \newtheorem*{lemma*}{Lemma}

\theoremstyle{definition}
    \mynewtheorem{definition}{Definition}
    \mynewtheorem{example}{Example}
    \newtheorem*{example*}{Example}

\theoremstyle{remark}
    \mynewtheorem{remark}{Remark}
    \mynewtheorem{notation}{Notation}

\newtheorem*{remark*}{Remark}

\newcommand{\A}{\alpha}
\newcommand{\B}{\beta}
\newcommand{\I}{\iota}
\newcommand{\D}{\Delta}

\newcommand{\N}{\Bbb N}
\newcommand{\Z}{\Bbb Z}
\newcommand{\C}{\Bbb C}
\newcommand{\R}{\Bbb R}

\renewcommand{\H}{\Bbb H}

\newcommand{\RP}{{\Bbb R\!\operatorname{P}}}
\newcommand{\CP}{{\Bbb C\!\operatorname{P}}}
\newcommand{\HP}{{\Bbb H\!\operatorname{P}}}

\newcommand{\T}{\otimes}
\newcommand{\e}{\varepsilon}
\newcommand{\p}{\varphi}
\newcommand{\W}{\bigwedge}
\newcommand{\rew}{\widetilde}
\newcommand{\w}{\wedge}

\newcommand{\Sq}{\operatorname{Sq}}

\newcommand{\Hom}{\operatorname{Hom}}
\newcommand{\Ext}{\operatorname{Ext}}

\newcommand{\GL}{{\operatorname{GL}}}

\renewcommand{\O}{{\operatorname{O}}}
\newcommand{\SO}{{\operatorname{SO}}}
\newcommand{\U}{{\operatorname{U}}}
\newcommand{\SU}{{\operatorname{SU}}}
\newcommand{\Sp}{{\operatorname{Sp}}}
\newcommand{\PSp}{{\operatorname{PSp}}}
\newcommand{\Spin}{{\operatorname{Spin}}}
\newcommand{\Spinc}{{{\operatorname{Spin}}^{\operatorname{c}}}}

\newcommand{\K}{\operatorname{K}}
\newcommand{\KU}{\operatorname{KU}}

\newcommand{\KO}{\operatorname{KO}}
\newcommand{\ku}{\operatorname{ku}}
\newcommand{\ko}{\operatorname{ko}}

\newcommand{\BH}{{\operatorname{BH}}}
\newcommand{\BG}{{\operatorname{BG}}}
\newcommand{\BO}{{\operatorname{BO}}}
\newcommand{\BSO}{{\operatorname{BSO}}}
\newcommand{\BU}{{\operatorname{BU}}}
\newcommand{\BSU}{{\operatorname{BSU}}}

\newcommand{\BPSp}{{\operatorname{BPSp}}}
\newcommand{\BSpin}{{\operatorname{BSpin}}}
\newcommand{\BSpinc}{{\operatorname{BSpin}^{\operatorname{c}}}}
\newcommand{\AC}{{{\alpha}^{\operatorname{c}}}}

\newcommand{\TC}{{T}^{\operatorname{c}}}
\newcommand{\MTC}{{{\mathcal T}^{\operatorname{c}}}}
\newcommand{\mTc}{{\m T}^{\operatorname{c}}}
\newcommand{\mT}{{\m T}}

\newcommand{\DC}{{\operatorname{D}^{\operatorname{c}}}}

\newcommand{\tc}{{\operatorname{t}^{\operatorname{c}}}}

\newcommand{\MO}{\operatorname{MO}}
\newcommand{\MSO}{\operatorname{MSO}}
\newcommand{\MU}{\operatorname{MU}}

\newcommand{\MSpin}{\operatorname{MSpin}}
\newcommand{\MSpinc}{{\operatorname{MSpin}^{\operatorname{c}}}}

\newcommand{\ol}{\overline}
\newcommand{\bd}{\bold}

\newcommand{\im}{\operatorname{im}}

\newcommand{\inv}{^{-1}}
\newcommand{\m}{\mathcal}
\newcommand{\f}{\mathfrak}

\newcommand{\hmod}{{/\!\!/}}

\renewcommand{\and}{\qquad\text{and}\qquad}

\newcommand{\pw}{,&\text{if }}
\newcommand{\op}[1]{{\operatorname{#1}}}

\def\RP{\mathbb{RP}}

\hypersetup{colorlinks=false,allbordercolors={red}}

\def\innerLmod(#1){{{}_{#1}\textbf{Mod}}}
\def\innerRmod(#1){{\textbf{Mod}_{#1}}}
\def\innerBimod(#1,#2){{{}_{#1}\textbf{Mod}_{#2}}}

\newcommand{\Lmod}{\text{-}\textbf{Mod}}

\newcommand{\copo}[2]{\left\langle#1,#2\right\rangle}

\newcommand{\jenny}[2]{\draw[fill=black] (#1,#2) circle (0.05);}
\newcommand{\leftcurved}[2]{\draw[semithick] (#1,#2) to[out=130,in=-130] (#1,#2+2);}
\newcommand{\rightcurved}[2]{\draw[semithick] (#1,#2) to[out=50,in=-50] (#1,#2+2);}
\newcommand{\upedge}[2]{\draw[semithick] (#1,#2) to[out=90,in=-90] (#1,#2+1);}
\newcommand{\ess}[2]{\draw[semithick] (#1,#2) to[out=0,in=190] (#1+2,#2+2);}

\newcommand{\leftcurvedlong}[2]{\draw[semithick] (#1,#2) to[out=120,in=-120] (#1,#2+3);}
\newcommand{\rightcurvedlong}[2]{\draw[semithick] (#1,#2) to[out=60,in=-60] (#1,#2+3);}

\newcommand{\diagramAone}[2]{
	\jenny{#1}{#2}\jenny{#1}{#2+1}\jenny{#1}{#2+2}\jenny{#1}{#2+3}
	\jenny{#1+2}{#2+3}\jenny{#1+2}{#2+4}\jenny{#1+2}{#2+5}\jenny{#1+2}{#2+6}
	\leftcurved{#1}{#2}
	\rightcurved{#1+2}{#2+4}
	\upedge{#1}{#2}\upedge{#1}{#2+2}
	\upedge{#1+2}{#2+3}\upedge{#1+2}{#2+5}
	\ess{#1}{#2+1}\ess{#1}{#2+2}\ess{#1}{#2+3}}
	
\newcommand{\diagramI}[2]{
	\jenny{#1}{#2}\jenny{#1}{#2+1}\jenny{#1}{#2+2}
	\jenny{#1+2}{#2+2}\jenny{#1+2}{#2+3}\jenny{#1+2}{#2+4}\jenny{#1+2}{#2+5}
	\ess{#1}{#2}\ess{#1}{#2+1}\ess{#1}{#2+2}
	\upedge{#1}{#2+1}\upedge{#1+2}{#2+2}\upedge{#1+2}{#2+4}
	\rightcurved{#1+2}{#2+3}}
	
\newcommand{\diagramJ}[2]{
	\jenny{#1}{#2-2}
	\jenny{#1}{#2-1}
	\jenny{#1}{#2}
	\jenny{#1}{#2+1}
	\jenny{#1}{#2+2}
	\upedge{#1}{#2-2}
	\upedge{#1}{#2+1}
	\leftcurved{#1}{#2-2}
	\leftcurved{#1}{#2}
	\rightcurved{#1}{#2-1}}
	
\newcommand{\diagramK}[2]{
	\jenny{#1}{#2}
	\jenny{#1}{#2+2}
	\jenny{#1}{#2+3}
	\upedge{#1}{#2+2}
	\leftcurved{#1}{#2}}
	
\newcommand{\diagramZ}[2]{\jenny{#1}{#2}}	

\newcommand{\diagramEone}[2]{
	\jenny{#1}{#2}
	\jenny{#1}{#2+1}
	\jenny{#1}{#2+3}
	\jenny{#1}{#2+4}
	\upedge{#1}{#2}
	\upedge{#1}{#2+3}
	\leftcurvedlong{#1}{#2}
	\rightcurvedlong{#1}{#2+1}}

\newcommand{\diagramL}[2]{
	\jenny{#1}{#2+1}
	\jenny{#1}{#2+3}
	\jenny{#1}{#2+4}
	\upedge{#1}{#2+3}
	\leftcurvedlong{#1}{#2+1}}
	
\newcommand{\diagramC}[2]{
	\jenny{#1}{#2}
	\jenny{#1}{#2+3}
	\leftcurvedlong{#1}{#2}}

\settitle{Scalar Curvature and Transfer Maps in $\Spin$ and $\Spinc$ Bordism} 
\setauthor{Elliot Granath} 
\setdoctype{Dissertation} 
\setdegree{Doctor of Philosophy} 
\setgradyear{2023} 
\setgradmonth{June} 
\setdept{Department of Mathematics} 

\begin{document}



   \clearpage
   \thispagestyle{empty}
   {%
   \centering
   \begin{DoubleSpace} \MakeUppercase{\thesistitle}

   \ifdraftdoc
   (Draft generated \today)
    \else
   \relax
    \fi
   \end{DoubleSpace}
   
    \begin{vplace}[1]
    \begin{SingleSpace}
        by \\ \vspace{\onelineskip}
        \MakeUppercase{\thesisauthor}
    \end{SingleSpace}
    \end{vplace}

   \begin{SingleSpace}
   \MakeUppercase{A \thesisdoctype} \\ \vspace{\onelineskip}
   Presented to the \thesisdept \\
   and the Division of Graduate Studies of the University of Oregon \\
   in partial fulfillment of the requirements \\
   for the degree of \\
   \thesisdegree \\ \vspace{\onelineskip}
    \thesisgraddate
   \end{SingleSpace}
   }
   \enlargethispage{\bottafiddle}
   \clearpage 

  
    \clearpage
{%
   \thispagestyle{plain}
    \thesisapproval
   
   \raggedright

   \begin{SingleSpace}
   \noindent
   Student: \thesisauthor \\ \vspace{\onelineskip}
   \begin{SingleSpace}
   \noindent Title: \thesistitle 
   \end{SingleSpace}

   \vspace{\onelineskip}

   \noindent This \thesisdoctype \ has been accepted and approved in partial fulfillment of the requirements for the \thesisdegree \ degree in the \thesisdept \ by:

   \vspace{\onelineskip}

   \noindent %
   \begin{tabular}{@{}ll@{}}
   Boris Botvinnik \hspace{1in}    & Chairperson \\
   Nicolas Addington                & Core Member \\
   Robert Lipshitz                & Core Member \\
   Peng Lu                & Core Member \\
   Spencer Chang               & Institutional Representative \\
   & \\
   and & \\
   & \\
   Krista Chronister                 & Vice Provost for Graduate Studies
   \end{tabular}

   \vspace{\onelineskip}

    \noindent Original approval signatures are on file with the University of Oregon Division of Graduate Studies.   
    \end{SingleSpace}
   
   \vspace{\onelineskip}

   \noindent Degree awarded \thesisgraddate.
   \enlargethispage{\bottafiddle}
   \clearpage
}



\pagestyle{plain}

\begin{center}
\vspace*{-\uppermargin}
\vspace*{4.5in}
\textcopyright \ \thesisgradyear \  \thesisauthor\\
This work is licensed under a Creative Commons\\
\textbf{Attribution License.}\\
\vspace{15pt}
\includegraphics{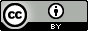}
\end{center}


   \clearpage
    {%
   \thispagestyle{plain}
   \thesisabstract

   \begin{DoubleSpace}
       
   \raggedright
   \noindent
   \thesisauthor \\ \vspace{\onelineskip}
   \noindent \thesisdegree \\ \vspace{\onelineskip}
   \noindent \thesisdept \\ \vspace{\onelineskip}
   \noindent \thesisgraddate \\ \vspace{\onelineskip}
   \noindent Title: \thesistitle 
   \end{DoubleSpace}
    }
   \vspace{\linespace}

In 1992, Stolz proved that, among simply connected $\Spin$-manifolds of dimension $5$ or greater, the vanishing of a particular invariant $\A$ is necessary and sufficient for the existence of a metric of positive scalar curvature. More precisely, there is a map $\A\colon\Omega_*^\Spin\to\ko$ (which may be realized as the index of a Dirac operator) which Hitchin established vanishes on bordism classes containing a manifold with a metric of positive scalar curvature. Stolz showed $\ker\A$ is the image of a transfer map $\Omega_{*-8}^\Spin\BPSp(3)\to\Omega_*^\Spin$. In this paper we prove an analogous result for $\Spinc$-manifolds and a related invariant $\AC\colon\Omega_*^\Spinc\to\ku$. We show that $\ker\AC$ is the sum of the image of Stolz's transfer $\Omega_{*-8}^\Spin\BPSp(3)\to\Omega_*^\Spinc$ and an analogous map $\Omega_{*-4}^\Spinc\BSU(3)\to\Omega_*^\Spinc$. Finally, we expand on some details in Stolz's original paper and provide alternate proofs for some parts.



\clearpage
   \thispagestyle{plain}
   \thesisacknowledgments
   
	Thank you to my advisor Boris Botvinnik and to my family, including mom, dad, Wes, Ashley, Reggie, Marge, Barbara, and Pancake.

   \newpage

   \thispagestyle{plain}
   \begin{vplace}[0.3] %
	For Grandma Marge
   \end{vplace} 
 
\tableofcontents




\pagestyle{uothesis}
 

\chapter{Introduction}

A major topic of differential geometry today is the study of Riemannian curvatures on a smooth manifold, and how this relates to topological invariants. Perhaps the simplest notion of curvature on a Riemannian manifold $(M,g)$ is the scalar curvature $s_g$. One way to define curvature is by comparing the volume of a geodesic ball $B_\e$ around a point $x\in M$ with a ball $B_\e^{(0)}$ of radius $\e$ in Euclidean space. More precisely, $s_g(x)$ is such that
  \[\frac{\op{Vol}(B_{\e})}{\op{Vol}(B_{\e}^{(0)})}=\left(1 -\frac{s_g(x)}{6(n+2)}\e^2+O(\e^4)\right).\]

In this document we address a fundamental question about Riemannian metrics of positive scalar curvature (psc metrics).
\begin{equation}\label{existence}
\text{Under which conditions does a manifold admit a psc metric?}
\end{equation}
The case for $2$-dimensional manifolds is unique: here the scalar curvature is simply twice the Gauss curvature $K_g$. The Gauss-Bonnet theorem states
\[\int_M s_g(x)d\sigma_g=4\pi\chi(M), \]
where $\sigma_g$ is the volume element corresponding to $g$ and $\chi(M)$ is the Euler characteristic of $M$. Hence, in $2$ dimensions, the existence of metrics of positive scalar curvature is a matter of which surfaces have positive Euler characteristic -- namely, $S^2$ and $\RP^2$.

The existence problem \ref{existence} has also been resolved for all manifolds of dimension five or greater which are simply connected. Under these conditions, Gromov and Lawson proved a crucial result of surgery theory. We recall some core definitions and facts of this field. By default, all manifolds we mention in this paper are compact.

A surgery on a manifold is defined as follows: let $S^p\times D^{q+1}\to M$ be an embedding of a sphere $S^p$ with a tubular neighborhood (a ``handle'') around it, where $p+q+1$ is the dimension of $M$. The surgery on $M$ given by this embedding produces a manifold $M'$ obtained by removing the interior of $S^p\times D^{p+1}\subset M$ and gluing in the handle $D^{p+1}\times S^q$:
\[M'=\left(M\setminus\op{int}\left(S^p\times D^{q+1}\right)\right)\cup_{S^p\times S^q}\left(D^{p+q}\times S^q\right). \]
In short, the effect of this surgery is to ``suture up'' the $p$-dimensional hole given by a topologically framed embedding of $S^p$. We say that the codimension of this surgery is $q+1$, since this is the codimension of $S^p$ in $M$. Gromov and Lawson established an important result about $\Spin$ manifolds and metrics of positive scalar curvature:
\begin{theorem}[\cite{gl}]\label{surgerylemma}
	Let $(M,g)$ be a compact Riemannian manifold with positive scalar curvature (a psc metric). If $M'$ is obtained from $M$ by a surgery of codimension $\ge 3$, then $M'$ carries a psc metric.
\end{theorem}

Theorem \ref{surgerylemma} has the following important corollary and related result which, for some classes of manifolds, allows us to only consider bordism classes when discussing the existence of psc metrics. For a review of $X$-bordism, see \cite[$\S 12.12$]{switzer}.
\begin{corollary}[\cite{gl}]\label{a1}
	Let $M$ be a simply connected $\Spin$ manifold of dimension $\ge 5$. If $M$ is $\Spin$-bordant to a manifold $M'$ with a psc metric, then $M$ also has a psc metric.
\end{corollary}
For manifolds which are not $\Spin$, Gromov and Lawson proved the following.
\begin{theorem}[\cite{gl}]\label{a2}
	Let $M$ be a simply connected manifold of dimension $\ge 5$ which is \textit{not} $\Spin$. Then $M$ admits a psc metric.
\end{theorem}
\begin{remark*}
	There is inconsistency in the literature on the usage of ``bordism'' versus ``cobordism.'' Following Rudyak (\cite[$7.19$]{rud}), we tend to use ``bordism'' ring and say that two manifolds are ``bordant.'' This terminology may be more common today since it allows authors to reserve ``cobordism'' for the dual cohomology theory (as in \cite[$7.30$]{rud} and \cite[$\S 12$]{switzer}). Historically the term ``cobordism'' was more common for the homology theory (as in \cite{stong}). This is presumably since ``cobordant'' has a literal interpretation of two manifolds which together bound another manifold (the French verb ``\textit{bordant}'' means ``to bound'').
\end{remark*}


We can summarize corollary \ref{a1} and theorem \ref{a2} and rephrase them as follows: for $M$ a simply connected manifold of dimension $n\ge 5$,
\begin{enumerate}
	\item[$\bullet$] if $M$ is $\Spin$, then the existence of a psc metric on $M$ depends only on the $\Spin$ bordism class $[M]\in\Omega_n^\Spin$;
	\item[$\bullet$] if $M$ is \textit{not} $\Spin$, then the existence of a psc metric on $M$ depends only on the oriented bordism class $[M]\in\Omega_n^\SO$.
\end{enumerate}

\begin{remark*}
	In particular, for simply connected manifolds of dimension $\ge 5$, the existence of a psc metric is determined by purely topological data. Similarly, a manifold $M$ is $\Spin$ if and only if $w_2(M)=0$ and simply connected if and only if $w_1(M)=0$ (assuming $M$ is connected). The equivalence classes of $\Spin$-structures on a manifold are in bijection with $H^1(M)$ (with coefficients in $\Z_2$), hence, if a simply connected manifold admits a $\Spin$-structure, the $\Spin$-structure is unique.
\end{remark*}

Since at least the early 1960s, it was known that the presence of a $\Spin$-structure on a manifold was had an intimate relationship with that of a psc metric. A Riemannian manifold $(M,g)$ which is $\Spin$ has a Dirac operator $D_g$ acting on the bundle of spinors over $M$, and in 1963 Lichnerowicz (\cite{lich}) proved the formula
\begin{equation}\label{eq-dirac}
	D_g^2=\nabla^*\nabla+\frac14s_g,
\end{equation}
where $\nabla$ is a covariant derivative and $\nabla^*$ is the formal adjoint. Equation (\ref{eq-dirac}) implies that if $s_g$ is positive, $D_g$ is invertible (i.e., there are no harmonic spinors on $M$). If $M$ is a $\Spin$ manifold of dimension $n$ with a psc metric, the Lichnerowicz formula implies that $[M]\in\Omega_n^\Spin$ lies in the kernel of the index
\begin{equation}\label{eq-alpha}
	\A\colon\Omega_n^\Spin\to\ko_n.
\end{equation}
Following Stolz's work, we had a complete answer to the existence problem \ref{existence} for simply connected manifolds of dimension $\ge 5$. Under these conditions,
\begin{itemize}
	\item [(i)] if $M$ is not $\Spin$, then $M$ admits a psc metric (\cite{gl});
	\item [(ii)] if $M$ is $\Spin$, then $M$ has a psc metric if and only if $\A([M])=0$ (\cite{stolz}).
\end{itemize}
\begin{remark*}
	In \cite{gl}, Gromov and Lawson proved the rational version of Stolz's result, which can be stated as follows: for a simply connected $\Spin$-manifold of dimension $n\ge 5$, if $\A([M])=0$, then some multiple $M\#M\#\cdots\#M$ admits a psc metric.
\end{remark*}

The techniques used by Stolz in his proof are central to the results in this paper. We now describe the geometric idea behind Stolz work. Let $\HP^2$ be the quaternionic projective space with the standard metric $g_0$. It is well known that the group of isometries of $\HP^2$ (with respect to $g_0$) is the projective group $\PSp(3)$ (which is the quotient of $\Sp(3)$ by $\{\pm1\}$). We now make an elementary observation.



\begin{lemma}
	Let $E$ be the total space of a smooth bundle over a manifold $M$ with fiber $\HP^2$ and structure group $\PSp(3)$. Then $E$ carries a psc metric.
\end{lemma}
\begin{proof}
	Cover $M$ by open sets $U_1,\ldots, U_n$ over which $E$ is locally (diffeomorphic to) $D_i\times \HP^2$. Let $g_M$ be the the metric on $M$ and $g_0$ be the standard metric on $\HP^2$. Since $M$ is compact, $s_{g_M}$ bounded, whence there exists sufficiently small $\lambda>0$ such that $g\coloneqq g_M\times \lambda g_0$ has positive scalar curvature on each neighborhood: $s_g=s_{g_M}+\lambda\inv s_{g_0}$. Again for $\lambda$ sufficiently small, this holds over each of the $U_i$. Finally, since $\PSp(3)$ acts isometrically on $\HP^2$, these local metrics glue to form a global psc metric. 
\end{proof}

\begin{remark*}
	It is an elementary fact that, for any $\lambda>0$, the scaled metric $\lambda g$ is a Riemannian metric with scalar curvature $s_{\lambda g}=\lambda\inv s_g$. In fact, since scalar curvature can be defined as a sum of sectional curvature, one's geometric intuition for surfaces (roughly) shows this for higher dimensions. 
\end{remark*}

%

Of course $\PSp(3)$ is a compact Lie group, so we have a classifying space $\BPSp(3)$ with a universal $\PSp(3)$-bundle and the associated $\HP^2$-bundle. We now define a transfer map
\begin{equation}\label{eq-trans}
	\Psi\colon\Omega_n^\Spin\BPSp(3)\to\Omega_{n+8}^\Spin
\end{equation}
as follows: a class in $\Omega_n^\Spin\BPSp(3)$ is represented by a pair $(M,f)$, where $M$ is an $n$-dimensional $\Spin$ manifold and $f$ is a map $M\to\BPSp(3)$. Let $\widehat M$ be the total space of the pullback via $f$ of the $\HP^2$-bundle. Note that $\widehat M\to M$ is an $\HP^2$-bundle which also has the structure group $\PSp(3)$ acting isometrically on the fibers. Now we have the following main result from \cite{stolz}.
\begin{theorem}[Stolz]\label{stz}
	For the transfer map $\Psi\colon\Omega_n^\Spin\BPSp(3)\to\Omega_{n+8}^\Spin$	 and index map $\A\colon\Omega_n^\Spin\to\ko_n$, we have $\im\Psi=\ker\A$.
\end{theorem}

Since Lichnerowicz's formula implies $\im\Psi\subseteq\ker\A$, Stolz had to show $\ker\A\subseteq\im\Psi$. He achieved this by reducing to a homotopy-theoretic problem: to begin with, the transfer $\Psi$ may be considered as the map induced on homotopy of a map of spectra
\begin{equation}\label{eq-trans2}
	T\colon\MSpin\w\Sigma^8\BPSp(3)_+\to\MSpin.
\end{equation}
The index map $\A$ is similarly induced by a map
\begin{equation}
	D\colon \MSpin\to\ko.
\end{equation}
The algebraic fact that $\A([M])=0$ if $M$ has a psc metric, for example, translates to statement that $DT$ is nullhomotopic. This means that $T$ factors through the homotopy fiber $\widehat\MSpin$ of $D$. We write $\widehat T$ for this lift and $i\colon\widehat\MSpin\to\MSpin$ for the inclusion of $\widehat\MSpin$. Theorem \ref{stz} then follows from showing $\widehat T$ induces a surjection of homotopy groups.

It is worth emphasizing here that, prior to Stolz's result in \cite{stolz}, the same result had been proven rationally in \cite{gl}. As was well known by that time, the map $\MSpin\to\MSO$ is a homotopy equivalence for any coefficient ring containing $1/2$. Hence the most difficult component is the $2$-primary data. Hence, in this paper, we work entirely with coefficients in $\Z_2=\Z/(2)$ for homology and cohomology. We may also abuse notation and write $\pi_*X$ to mean $\pi_*(X)\T\Z_{(2)}$.
 
\section{Twisted scalar curvature for $\Spinc$ manifolds}

As we have seen, existence of a $\Spin$ structure (and consequently of the Dirac operator) plays a fundamental role in the existence of a psc metric. More generally, a $\Spinc$-structure implies the existence of a $\Spinc$ Dirac operator.
\\

Let $M$ be a simply-connected manifold with
  $w_2(M)\neq 0$ (i.e. $M$ is not a spin manifold). Then $M$ has a
  $\Spinc$-structure if there is a class $c\in H^2(M;\Z)$ which maps
  to $w_2(M)$ under the mod-2 reduction map $H^2(M;\Z)\to
  H^2(M;\Z_2)$.  The class $c$ gives a map $c: M\to \CP^{\infty}$ and,
  consequently, a complex line bundle $L \to M$. We use the notation
  $(M,L)$ for a manifold with a choice of $\Spinc$-structure.  We
  notice that a $\Spin$-manifold $M$ has a canonical $\Spinc$ corresponding to the trivial complex line bundle over $M$.
\\

Assume $(M,L)$ is a non-$\Spin$ $\Spinc$ manifold. We
  choose a Riemannian metric $g$ on $M$, a Hermitian metric $h$ on
  $L$, and a unitary connection $A_L$ on the line bundle $L$. These data
  give the $\Spinc$ Dirac operator $D_{(M,L)}$. Then we have the Lichnerowicz formula \cite{lich}
  \begin{equation}\label{eq-dirac-c}
  D_{(M,L)}^2 = \nabla\nabla^* + \frac{1}{4}s_g + \m R_L
  \end{equation}  
here
\begin{equation*}
 \m{R}_L =\frac12\sum_{j<k}F_L(e_j,e_k)\cdot e_j\cdot e_k,
\end{equation*}
where one sums over an orthonormal frame, and where $F_L$ is the
curvature of the connection $A_L$ on the line bundle $L$. We denote
\begin{equation*}
s^L_{g}:=s_{g} + 4\mathcal{R}_L,
\end{equation*} 
and say that $s^L_{g}$ is the \emph{$L$-twisted scalar curvature}. Notice that $s^L_{g}$ is, in fact, a zeroth order operator on spinors, and it depends on a choice of the hermitian metric $h$ on $L$ and on the connection $A_L$.

\section{Existence of positive $L$-twisted
    scalar curvature}
    
One of our primary goals is to resolve the following existence question: \emph{under which conditions on a non-$\Spin$ $\Spinc$ manifold $(M,L)$ does there exist a Riemannian metric $g$, a Hermitian metric $h$ on $L$, and a connection $A_L$ such that the $L$-twisted scalar curvature $s_g^{L}$ is positive?}

It turns out that the positivity of the $L$-twisted
  scalar curvature $s_g^{L}$ is also invariant under surgeries of
  codimension at least three. In particular, this implies that for a
  simply-connected manifold $M$ the existence of the data $(g,h,A_L)$
  as above such that $s_g^{L}$ is a positive operator, depends only on
  the $\Spinc$-cobordism class $[(M,L)]\in \Omega^{\Spinc}$.

Next, we recall that the index of the $\Spinc$ Dirac
  operator $D_{(M,L)}$ takes values in $\ku_*$, where $\ku$ is the connective cover of the complex $K$-theory spectrum $\KU$. In particular, we have a $\Spinc$-version of the $\alpha$ invariant
   \begin{equation}\label{eq-alpha-c}
\AC : \Omega^{\Spinc}_n\to\ku_n.
  \end{equation} 
The formula (\ref{eq-dirac-c}) shows that, if the $L$-twisted scalar
curvature $s_g^{L}$ is positive, then the $\Spinc$-bordism class
$[(M,L)]\in \Omega^{\Spinc}_n$ is in the kernel of the index map
$\AC$. This leads to the following existence result due to
Botvinnik and Rosenberg.
\begin{theorem}[\cite{br22, br18}]\label{BR1}
Let $(M,L)$ be a simply connected non-$\Spin$ $\Spinc$ manifold of
dimension $n\geq 5$. Then $M$ admits a Riemannian psc metric $g$, a
hermitian metric $h$ and a connection $A_L$ such that $s_{g}^L >0$ if
and only if the $\Spinc$ index $\AC[(M,L)]$ vanishes in $\ku_n$.
\end{theorem}

A proof of Theorem \ref{BR1} is given in \cite{br22} Theorem 3.8
and \cite{br18} Corollary 32 which relies on deep results in
cobordism and homotopy theory. A main goal of this thesis is to give a
direct proof of Theorem \ref{BR1} using the same technology developed by Stolz in \cite{stolz}.

As with the $\Spin$ case, we will ignore odd-primary data. The proof for the case of odd primes also appears in \cite{br22}; in this paper we give a direct, constructive proof for the $2$ primary case.

\section{New transfer map and main result}

We notice that the complex projective spaces $\CP^{2k}$ are non-spin manifolds. In
particular, $\CP^2$ has a standard $\Spinc$-structure given by a
complex line bundle $L_0\to \CP^2$ with $c_1(L_0)=x$, where $x$ is a
generator of $H^2(\CP^2;\Z)$. A calculation of
$\AC[(\CP^{2k},L_0)]$ was worked out by Hattori \cite{hat},
in particular, he shows that $\AC[(\CP^2,L_0)]=0$.

The projective plane $\CP^2$ has a remarkable property -- namely, the group $G=\SU(3)$ acts transitively on $\CP^2$, and $\CP^2= G/H$, where the subgroup $H\coloneq\SU(2,1)$ is the subgroup of elements in $\U(2)\times\U(1)\subset\U(3)$ with determinant $1$. We obtain a fiber bundle $p: \BH\to \BG$ with fiber $\CP^2$ and structure group $\SU(3)$. Thus given a $\Spinc$-manifold $(M, L)$ and a map $f: M\to BG$, we can form the associated $\CP^2$-bundle $\hat p: E\to M$ as a pull-back


\begin{figure}[h!]\centering
	\begin{tikzcd}
		E \ar[d,"\widehat p"] \ar[r] & \BH \ar[d,"p"]\\
		M \ar[r,"f"] & \BG 
	\end{tikzcd}
\end{figure}
\noindent
where $E=M\times_f \CP^2$ has dimension $n+4$ and has
  a $\Spinc$ structure inherited from the $\Spinc$ structure on $M$
  defined by $L$ and the $\Spinc$ structure on $\CP^2$ defined by the
  bundle $L$. This construction defines a transfer map
\begin{equation}\label{mtc}
\MTC: \Omega_{n-4}^{\Spinc}(\BSU(3)) \to \Omega_{n}^{\Spinc}. 
\end{equation}
We also have a $\Spinc$-version of the Stolz transfer map:
\begin{equation}\label{mt}
\m T: \Omega_{n-8}^{\Spinc}(\BPSp(3)) \to \Omega_{n}^{\Spin}. 
\end{equation}
Here is the main result of this thesis:
\begin{theorem}\label{mymain}
  The transfer maps $\m T$ and $\MTC$ are such that
  \[\im(\m T)+\im(\MTC)=\ker \AC\]
  as abelian groups.
\end{theorem}

\section{Proof of main result (theorem \ref{mymain})}

We summarize and expand on Stolz's original proof and give an analogous result for $\Spinc$ bordism involving the transfer map $\mTc\colon\Omega_{*-4}^\Spinc\BSU(3)\to\Omega_*^\Spinc$ of (\ref{mtc}). The Steenrod algebra $\m A$ and basic theory of modules and comodules over $\m A$ (and its dual) are needed to describe the transfer maps, so we begin with some algebraic preliminaries in chapter $2$. In chapter $3$ we give more details about the transfer maps $\m T$ and $\mTc$ and compute their induced maps on cohomology. Important lemmas and other results used in the proof of theorem \ref{mymain} appear in chapter $3$, with some necessary computations included in the appendices.

The compact Lie group $\PSp(3)$ acts transitively on $S^{11}\subset\H^3$ and this descends to a transitive action on $\HP^2$. The stabilizer of $[0:0:1]$ is $\PSp(2,1)\coloneqq \op{P}(\Sp(2)\times\Sp(1))$, giving a fiber bundle $\PSp(2,1)\to\PSp(3)\to\HP^2$. In turn this yields a bundle $\pi\colon\BPSp(2,1)\to\BPSp(3)$ with fiber $\HP^2$ and structure group $\PSp(3)$. Note that the standard metric on $\HP^2$ has positive scalar curvature and the structure group $\PSp(3)$ acts via isometries with respect to the standard metric.

Using observations in the last paragraph we obtain the transfer map of Stolz $\m T\colon\Omega_{n-8}^\Spin\BPSp(3)\to\Omega_n^\Spin$ as follows: a class in $\Omega_n^\Spin\BPSp(3)$ can be written $[M,f]$, where $M$ is an $n$-dimensional $\Spin$ manifold and $f$ is a map $M\to\BPSp(3)$. Let $\widehat M$ be the total space of the pullback via $f$ of our $\HP^2$-bundle. Now $[\widehat M]$ is a class in $\Omega_{n+8}^\Spin$, and the transfer map $\m T$ takes $[M,f]$ to $[\widehat M]$. Now the natural inclusion $\MSpin\to\MSpinc$ lets us easily translate this to $\Spinc$-bordism, although we abuse notation slightly and reuse the map names, e.g. $\m T\colon\Omega_{n-8}^\Spinc\BPSp(3)\to\Omega_n^\Spinc$. This notation is different from that of Stolz, who writes $\Psi$ instead of $\m T$.

Analogously to the quaternionic case, the special unitary group $\SU(3)$ acts transitively on $S^5\subset\C^3$ yielding a transitive action on $\CP^2$. As before, the stabilizer of $[0:0:1]$ is $\SU(2,1)\coloneq\op{S}(\U(2)\times\U(1))$ and we hence obtain a bundle $\BSU(2,1)\to\BSU(3)$ with fiber $\CP^2$. This bundle admits a $\Spinc$ structure, giving a map $\BSU(3)\to\BSpinc$. In this case we get a transfer map $\mT\colon\Omega_{n-4}^\Spinc\BSU(3)\to\Omega_n^\Spinc$ with the property $\im \TC\subseteq \ker\AC$. Unfortunately, the analogy with the $\Spin$ case ends here, as the reverse containment fails. Moreover, we will see that neither $\m T$ nor $\TC$ are surjective onto $\ker\AC$; instead we combine the two transfers and show $\im\m T+\im\TC=\ker\AC$.

Via the classical Pontrjagin-Thom construction $\Omega_n^\Spinc(X)$ can be identified with $\pi_n(\MSpinc\w X_+)$. The transfer map $\m T$ is the map on homotopy groups induced by a map of spectra $T\colon\MSpinc\w\Sigma^8\BPSp(3)_+\to\MSpinc$; similarly $\MTC$ is induced by a map $\TC\colon\MSpinc\w\Sigma^4\BSU(3)\to\MSpinc$. The $\AC$ invariant corresponds to a spectrum map $\DC\colon\MSpinc\to\ku$, and we have the following diagram of maps, where $i\colon\widehat{\MSpinc}\to \MSpinc$ is inclusion of the homotopy fiber of $\DC$ and $\mu\colon\MSpinc\w\MSpinc\to\MSpinc$ is the ring spectrum product map.



\begin{figure}[h!]\centering
	\begin{tikzcd}
		& & \widehat{\MSpinc}\ar[d,"i"] &\\
		\MSpinc\w\left(\Sigma^8\BPSp(3)_+\vee\Sigma^4\BSU(3)_+\right)\ar[r]\ar[urr,"\widehat T\vee\widehat \TC"]\ar[rr,"T\vee \TC" below,bend right=15] & \MSpinc\w\MSpinc\ar[r,"\mu"] & \MSpinc\ar[d,"\DC"]\\
		&&\ku
	\end{tikzcd}
\end{figure}
The unlabeled map is $1\w(t\vee \tc)$, where $t$ and $\tc$ are the respective bundle transfer maps composed with the Thom map. That is,
\begin{itemize}
  \item $t\colon\Sigma^8\BPSp(3)_+\to\MSpinc$ is the composition $t=\m M\circ\m T$;
  \item $\m T\colon\Sigma^8\BPSp(3)_+\to M(-\tau)$ is the bundle transfer map of \cite[\S VI.6]{board},
  \begin{itemize}
  \item[$\bullet$] where $-\tau$ is the stable complement of the bundle of tangent vectors along the fibers of $\BPSp(2,1)\to\BPSp(3)$;
\end{itemize}
  \item $\m M\colon M(-\tau)\to\MSpinc$ is the Thom map induced by the classifying map of $-\tau$.
\end{itemize}
Our main result can now be restated as follows:

\begin{theorem}\label{mainsurj}
	On homotopy groups, $T\vee \TC$ induces a surjection
	\[\Omega_*^\Spinc\left(\Sigma^8\BPSp(3)_+\vee\Sigma^4\BSU(3)_+\right)\to\pi_*\widehat\MSpinc.\]
\end{theorem}

\begin{remark}
	After localizing at $2$, the spectrum $\MSpinc$ splits into a wedge of suspensions of $\ku$ with other spectra. The map $\DC$ is simply projection onto the lowest degree copy of $\ku$. Details can be found in \cite[\S 5]{hophov}. The analogous statements hold for $\MSpin$ and $\ko$. These facts give a particularly nice description of $\widehat \MSpinc$ and $\widehat \MSpin$ which make these homotopy fibers reasonable objects to work with.
\end{remark}

To prove theorem \ref{mainsurj} (at the prime $2$), we use the Adams spectral sequence along with homological data. More precisely, we will show that $\widehat T\vee\widehat\TC$ induces a split surjection of $\m A_*$-comodules, hence a surjection of the $E_2$-terms
\begin{equation}\label{e2surj}
	\Ext_{\m A_*}^{s,t}(\Z_2,H_*\MSpinc\T(H_{*-8}\BPSp(3)\oplus H_{*-4}\BSU(3)))\to
		\Ext_{\m A_*}^{s,t}(\Z_2,H_*\widehat\MSpinc)
\end{equation}

\noindent which converge to
\begin{equation}\label{einfsurj}
	\pi_*(\MSpinc\w(\Sigma^8\BPSp(3)_+\vee\Sigma^4\BSU(3)_+))\T\Z_{(2)}\to \pi_*(\widehat\MSpinc)\T\Z_{(2)}
\end{equation}

It turns out that the first of these spectral sequences has no nontrivial differentials, hence the surjection (\ref{e2surj}) at the level of $E_2$ pages implies the surjection (\ref{einfsurj}) on the $E_\infty$ pages.

To prove theorem \ref{mymain}, then, it remains to show $\widehat T_*\vee\widehat\TC_*$ is a split surjection of $\m A_*$-comodules. Several algebraic maneuvers make our task easier: a useful change-of-rings construction reduces our task to working with comodules over a sub-coalgebra $E(1)_*\subset \m A_*$. We also prefer to dualize to cohomology in order to work with $\m A$-modules and $E(1)$-modules rather than comodules, although most results are stated in terms of comodules where possible. Finally, we give an algebraic result in Theorem \ref{123} which provides an easy way to show that an injection of $E(1)$-modules (or a surjection of $E(1)_*$-comodules) is split. Following \cite{pet}, we also know
\[H_*\BSpinc=\Z_2[x_i^2,x_j:\A(i)<3,\A(j)\ge 3], \]
where $\A(n)$ is the number of nonzero terms in the base-$2$ expansion of $n$.

\begin{theorem}\label{square}
	We have a commuting diagram as shown, where $\rho\colon \BSpinc\to\BO$ is the projection and $D$, $U$ and $\I$ are orientation classes. We make some observations about this diagram.
	\begin{figure}[h!]\centering
	\begin{tikzcd}
		\MSpinc\ar[r,"D"]\ar[d,"M\!\rho"] & \ku\ar[d,"\I"]\\
		\MO\ar[r,"U"]& H\Z_2
	\end{tikzcd}
	\end{figure}
We have the standard identification $H_*H\Z_2=\Z_2[\zeta_1,\zeta_2,\ldots]$, and, via the Thom isomorphism, we can write $H_*\MO=\Z_2[x_1,x_2,\ldots]$, where $\deg(x_i)=i$. The spectrum $\MO$ is $H\Z_2$-oriented via a Thom class (i.e., an orientation) $\m U\colon \MO\to H\Z_2$. Similarly $\DC\colon \MSpinc\to\ku$ can be considered as the $\ku$-orientation of $\MSpinc$. We use the Thom isomorphism to identify $H_*\MO$ with $H_*\BO=\Z_2[x_1,x_2,x_3,\ldots]$, where $\deg(x_i)=i$. The dual Steenrod algebra $\m A_*$ is the polynomial ring $\Z_2[\xi_1,\xi_2,\ldots]$, although we prefer to use the Hopf conjugate generators and write $\m A_*=\Z_2[\zeta_1,\zeta_2,\ldots]$ (here $\deg(\xi_i)=\deg(\zeta_i)=2^i-1$). Of course $\m A_*=H_*H\Z_2$, and the inclusion $\ko\to H\Z_2$ induces on homology the injection
\[\Z_2[\zeta_1^2,\zeta_2^2,\zeta_3,\zeta_4,\ldots]\to\Z_2[\zeta_1,\zeta_2,\zeta_3,\zeta_4,\ldots].\]
	\begin{enumerate}
	  \item The maps $M\!\rho$ and $\I$ induce monomorphisms on homology.
	  \item We can identify $H_*\MSpinc=\im(M\! \rho_*)=\Z_2[x_i^2,x_j:\A(i)<3,\A(j)\ge 3]$.
	  \item Define $P\subset H_*\MO$ as $P=\Z_2[x_{2^i-1}:i\ge 1]$. Then $U_*$ maps $P$ isomorphically onto $H_*H\Z_2$.
	  \item Identifying $P\subset H_*\MSpinc$, let
	  \[R=P\cap H_*\MSpinc=\Z_2[x_1^2,x_3^2,x_7,x_{15},\ldots]\subset H_*\MSpinc.\] Then $D_*$ maps $R$ isomorphically onto $H_*\ku$.
	\end{enumerate}
\end{theorem}
\begin{proof}
	The first claim is immediate from the fact that the corresponding maps on cohomology are quotients (see \cite{stong}, for example). The computation of $H_*\MSpinc$ and $H^*\MSpinc$ can be found in \cite{switzer}. The third claim is \cite[Corollary $4.7$]{stolz}, and the last follows from the nontrivial fact that $D_*$ admits a splitting $H_*\ku\to H_*\MSpinc$ (\cite{hat}, \cite{stolz2}).
\end{proof}


There is a useful functorial construction for $\MSpinc$-module spectra which identifies $H_*X$ with $\m A_*\square_{E(1)_*}\ol{H_*X}$, where $\ol{H_*X}\coloneqq\Z_2\T_R H_*X$ is a subalgebra of $H_*\MSpinc$ which maps isomorphically to $H_*\ku$ via $\DC_*$. The algebra $E(1)$ is generated by $Q_0=\Sq^1$ and $Q_1=\Sq^1\Sq^2+\Sq^2\Sq^1$; the cotensor product $\square$ is also standard notation and also defined in chapter $2$. (Note that $R=\m A_*\square_{E(1)_*}\Z_2$.) This construction, which is detailed in theorem \ref{reducedc}, is functorial in the sense that maps $f\colon H_*X\to H_*Y$ induce maps $\ol f\colon\ol{H_*X}\to \ol{H_*Y}$. This method allows us to work with modules over $E(1)$ (or comodules over $E(1)_*$) rather than using the whole Steenrod algebra.
\\

Applying this construction to the diagram which appeared just before theorem \ref{mainsurj} yields the following diagram.
\pagebreak

\begin{figure}[h!]\centering
	\begin{tikzcd}
		\ol{H_*\MSpinc}\T\left(H_*\Sigma^8\BPSp(3)_+\oplus H_*\Sigma^4\BSU(3)_+\right)\ar[d,"1\T\left(t_*\oplus \tc_*\right)"]\\
		\ol{H_*\MSpinc}\T H_*\MSpinc \ar[d,"\mu"]\\
		\ol{H_*\MSpinc}\ar[d,"\ol D_*"]\\
		\ol{H_*\ku}
	\end{tikzcd}
\end{figure}

An important observation is that $\ol{H_*\ku}=\Z_2$, and we can identify $\ol D_*$ with the augmentation map of $\ol{H_*\MSpinc}$. We can summarize these maps in more compact notation by writing $H_1=H_*\Sigma^8\BPSp(3)_+$ and $H_2=H_*\Sigma^8\BSU(3)_+$.

\begin{figure}[h!]\centering
	\begin{tikzcd}
		\ol{H_*\MSpinc}\T\left(H_1\oplus H_2\right) \ar[r] \ar[rr,"\ol T_*", bend left=15] & \ol{H_*\MSpinc}\T H_*\MSpinc \ar[r,"\mu"]\ar[d,"1\T p"] & \ol{H_*\MSpinc}\ar[r,"\ol D_*"] \ar[d,equal] & \ol{H_*\ku}\\
		& \ol{H_*\MSpinc}\T\ol{H_*\MSpinc} \ar[r,"m"] & \ol{H_*\MSpinc}
	\end{tikzcd}
\end{figure}

Recall that $\ol M=\Z_2\T_R M$, and here $p\colon M\to\ol M$ is $x\mapsto 1\T x$. In these terms $\mu$ is given by $(1\T x)\T y\mapsto 1\T xy$, while $m\circ(1\T p)$ is $(1\T x)\T y\mapsto (1\T x)\T(1\T y)\mapsto 1\T xy$. It follows that the diagram commutes and $\ol T_*$ is equal to the composition
\[\ol{H_*\MSpinc}\T\left(H_1\oplus H_2\right)\to \ol{H_*\MSpinc}\T\ol{H_*\MSpinc}\xrightarrow{m} \ol{H_*\MSpinc}. \]

\begin{lemma}\label{25}
	To show that $T_*$ is surjective onto $\ker D_*$, it suffices to show
	\[H_n\left(\Sigma^8\BPSp(3)_+\vee \Sigma^4BSU(3)_+\right)\to QH_n\MSpinc\]
	is surjective for $n\ge 4$ with $n\neq 2^k\pm 1$, where $QH_n\MSpinc$ is generated by indecomposable elements in $H_n\MSpinc$
\end{lemma}
\begin{proof}
First we claim that it suffices to show
\[H_*\left(\Sigma^8 \BPSp(3)_+\vee\Sigma^4\BSU(3)_+\right)\to Q\ol{H_*\MSpinc}\]
is surjective. This amounts to the obvious fact that $\ker \ol D_*$ is generated (over $\Z_2$) by elements in $\ol{H_*\MSpinc}$ with at least one indecomposable factor. These generators are of the form $\mu(x\T y)$ for some $x\in\ol{H_*\MSpinc}$ and $y\in Q\ol{H_*\MSpin}$ (note that $x$ may be $1$ here). Provided $y$ is in the image of $t_*$, $\mu(x\T y)$ is in the image of $\ol T_*$. The result now follows from Lemma \ref{q}.
\end{proof}

\begin{lemma}\label{q}
	The projection $p\colon H_*\MSpinc\to \ol{H_*\MSpinc}$ induces a homomorphism $Qp\colon QH_*\MSpinc\to Q\ol{H_*\MSpinc}$ such that $Qp\colon QH_n\MSpinc\to Q\ol{H_n\MSpinc}$ is an isomorphism of abelian groups for $n\ge 4$ with $n\neq 2^k\pm 1$. Further, $Q\ol{H_n\MSpinc}=0$ for $n<4$ or $n=2^k\pm 1$.
\end{lemma}
\begin{proof}
	Recall that $H_*\BSpinc=\Z_2[x_i^2,x_j:\A(i)<3,\A(j)\ge 3]$ and we defined the subring $R=\Z_2[p_1^2,p_2^2,p_3,p_4,\ldots]$. Here $p_i$ is the image of $x_{2^i-1}$ under the Thom isomorphism, so identifying $H_*\MSpin$ with $H_*\BSpin$ (as algebras), we can write $R=\Z_2[x_1^2,x_3^2,x_7,x_{15},\ldots]$. It is now easy to see that $\ol{H_*\BSpinc}$ has generators $x_i^2$ and $x_j$ of $H_*\BSpinc$ except for $x_1^2$, $x_3^2$, and $x_{2^n-1}$ for $n\ge 3$ (since $\A(2^n-1)=n$). This shows $Q\ol{H_n\BSpinc}=0$ for $n<4$ and $n=2^k-1$; to finish the lemma, use the ring identification $H^*\BSpinc=\Z_2[w_n:n\ge 2,n\neq 2^k+1]$ to show $Q\ol{H_n\MSpinc}$ is trivial for all $n=2^k+1$.
\end{proof}

\begin{remark*}
	We pause here to make an important note: as an $\m A$-module, $H^*\BSpinc$ is the quotient of $H^*\BSO$ by the ideal generated by $\Sq^3$ (\cite[\S XI]{stong}). As an algebra, one can show with the Ad\`em relations that $H^*\BSpinc$ is the polynomial ring with generators $w_i$ for all $i\ge 2$ with $i\neq2^k+1$. This identification does not respect the full $\m A$-module structure, however it respects the actions by $Q_0$ and $Q_1$. The analogous fact is true for the $\Spin$ case and the actions by $\Sq^1$ and $\Sq^2$ (\cite{stolz}). An explicit example of how this can cause issues is given in remark \ref{b5}.
\end{remark*}

After dualizing to cohomology, theorem \ref{mymain} now follows from the following lemma.
\begin{lemma}\label{37}
	The map \[PH^n\BSpinc\to H^n(\Sigma^8\BPSp(3)_+\vee\Sigma^4\BSU(3)_+) \]
is injective for $n\ge 4$, $n\neq 2^k\pm 1$.
\end{lemma}
\begin{proof}
	The proof of \cite[proposition $7.5$]{stolz} given by Stolz easily shows that the primitive generator $y_n$ maps injectively for $n\ge 8$, $\A(n)\ge 2$, since in these degrees $z_n$ is the image of $y_n$ under $H^*\BSpinc\to H^*\BSpin$ (the primitive generators $z_n$ and $y_n$ are detailed in appendix B). The remaining primitives are $y_6$ and $y_{2^k}$ for $k\ge 2$. Since $\A(6)=2$, $y_6=s_{3,3}=w_2^3+w_2w_4+w_6$. Of course $\A(2^k)=1$, so $y_{2^k}=w_2^{2^{k-1}}$. In either case, $y_n$ maps nontrivially, as witnessed by $w_2$ following theorem \ref{thm}.
\end{proof}

Lemmas \ref{37} and \ref{25} together prove theorem \ref{mainsurj} and thus complete the proof of theorem \ref{mymain}.

\chapter{Algebraic Background}

In this section we provide an algebraic overview of the Steenrod algebra and its dual. We then discuss certain subalgebras and quotient algebras relevant to our purposes. We will assume all (co)homology has coefficients in $\Z_2$ unless otherwise stated.

\section{The Steenrod algebra and its dual}

We denote the Steenrod algebra as $\m A\coloneqq H^*H\Z_2$. As an algebra, $\m A$ is generated by the Steenrod squares $\Sq^n$ subject to the Adem relations: for $m<2n$,
\[\Sq^m\Sq^n=\sum_{0\le i\le\lfloor m/2\rfloor}\binom{n-i-1}{m-2i}\Sq^{m+n-i}\Sq^i. \]
Now $\m A$ is a Hopf algebra, and the coproduct $\D\colon\m A\to\m A\T\m A$ is characterized as the algebra homomorphism satisfying the Cartan formula $\D(\Sq^k)=\sum_{i+j=k}\Sq^i\T\Sq^j$. We write $\m A_*$ for the dual algebra. Milnor established that $\m A_*$ is the polynomial ring (over $\Z_2$) with generators $\xi_i$ of degree $2^i-1$ for all $i>0$ (we also use the convention $\xi_0=1$). We will abuse notation and reuse $\D$ to denote the coproduct in $\m A_*$:
\[\D(\xi_k)=\sum_{i+j=k} \xi_i^{2^j}\T\xi_j.\]
Similarly, we use $\mu$ for the product and $\chi$ for the conjugate map in either case. The latter is characterized by commutativity of the following diagrams.
\begin{figure}[h!]\centering
	\begin{tikzcd}[column sep={1.25cm,between origins}, row sep={1.6cm,between origins}]
		&\m A\T \m A \ar[rr,"1\T\chi"] && \m A\T \m A\ar[dr] &\\
		\m A\ar[ur]\ar[rr]\ar[dr] && \Z_2\ar[rr] && \m A\\
		&\m A\T \m A \ar[rr,"\chi\T1"] && \m A\T \m A\ar[ur] &
	\end{tikzcd}
	\qquad
	\begin{tikzcd}[column sep={1.25cm,between origins}, row sep={1.6cm,between origins}]
		&\m A_*\T \m A_* \ar[rr,"1\T\chi"] && \m A_*\T \m A_*\ar[dr] &\\
		\m A_*\ar[ur]\ar[rr]\ar[dr] && \Z_2\ar[rr] && \m A_*\\
		&\m A_*\T \m A_* \ar[rr,"\chi\T1"] && \m A_*\T \m A_*\ar[ur] &
	\end{tikzcd}
\end{figure}

The Steenrod algebra acts naturally on the $\Z_2$-cohomology of any space and satisfies the following properties: for a space $X$ with $x,y\in H^*X$
\begin{enumerate}
  \item $\Sq^n(x)=x^2$ if $\deg(x)=n$
  \item $\Sq^n(x)=0$ if $\deg(x)<n$
  \item $\Sq(xy)=\Sq(x)\Sq(y)$, where $\Sq=1+\Sq^1+\Sq^2+\cdots$.
\end{enumerate}

As we mentioned, third property is also called the Cartan formula. Note that $\Sq(x)$ is always a finite sum due to the second property. An instructive application of the properties of the $\Sq^n$ is when $x$ is the generator of $H^*\RP^\infty$. In this case $\Sq(x)=1+x+x^2$, and this completely determines the action of $\m A$ on $H^*\RP^\infty$. Via the splitting principle, one can use this to determine the action on $\BO$.
\\

As an algebra, $\m A$ is generated by the classes $\Sq^n$ for $n\ge 0$, but this is not a minimal generating set: for example, $\Sq^3=\Sq^1\Sq^2$. We demonstrate now that $\m A$ is generated (as an algebra) by only the classes $\Sq^{2^n}$ for all $n$. First we need a useful fact of arithmetic modulo $2$.

\begin{lemma}\label{binomialfact}
	Given positive integers $a,b$ with base-$2$ expansions $a=a_0+2a_1+\cdots+2^na_n$ and $b=b_0+2b_1+\cdots+2^nb_n$, the modulo-$2$ binomial coefficient $\binom{a}{b}$ factors as
	\[\binom{a_0+2a_1+2^2a_2+\cdots+2^na_n}{b_0+2b_1+2^2b_2+\cdots+2^nb_n}=\binom{a_0}{b_0}\binom{a_1}{b_1}\binom{a_2}{b_2}\cdots\binom{a_n}{b_n}. \]
\end{lemma}
\begin{proof}
	By definition or otherwise, $\binom{a}{b}$ is the coefficient of $x^b$ in $(1+x)^a$. Now the Frobenius map gives
	\begin{align*}
		(1+x)^a&=(1+x)^{a_0}(1+x)^{2a_2}(1+x)^{2^2a_2}\cdots(1+x)^{2^na_n}\\
		&=(1+x)^{a_0}\big(1+x^2\big)^{a_2}\big(1+x^{2^2}\big)^{a_2}\cdots\big(1+x^{2^n}\big)^{a_n}.
	\end{align*}
	It follows from the Euclidean division algorithm that the only way to obtain a $x^b$ term is if each factor $(1+x^{2^k})^{a_k}$ contributes $x^{2^kb_k}$. Hence $\binom{a}{b}$ is nonzero precisely if every factor $\binom{a_i}{b_i}$ is nonzero.
\end{proof}

Notice that a factor $\binom{a_i}{b_i}$ is $0$ if and only if $a_i=0$ and $b_i=1$. Hence a useful interpretation of lemma \ref{binomialfact} is that $\binom{a}{b}=1$ if and only if every digit in the base-$2$ expansion of $a$ is greater than or equal to the corresponding digit in the base-$2$ expansion of $b$.

\begin{theorem}
	As an algebra, $\m A$ is generated by the classes $\Sq^{2^n}$ for all $n$.
\end{theorem}
\begin{proof}
	If $n$ is not a power of $2$, pick the largest integer $k$ such that $2^k<n$. Let $b=2^k$ and $a=n-b$. Then
	\[\Sq^a\Sq^b=\sum_c\binom{b-c-1}{a-2c}\Sq^{a+b-c}\Sq^c. \]
	When $c=0$, we have the term $\binom{b-1}{a}\Sq^{a+b}\Sq^0=\binom{b-1}{a}\Sq^n$. Now $b-1=2^k-1$ has the base $2$ expansion $1+2+2^2+\cdots+2^{k-1}$, and, by construction, $a=n-b$ can only contain summands $2^i$ with $i\le k-1$. By corollary (\ref{binomialfact}), $\binom{b-1}{a}=1$, and thus
	\[\Sq^a\Sq^b=\Sq^n+\sum_{c>0}\binom{b-c-1}{a-2c}\Sq^{a+b-c}\Sq^c. \]
	That is,
	\[\Sq^n=\Sq^{n-2^k}\Sq^{2^k}+\sum_{c>0}\binom{2^k-c-1}{n-2^k-2c}\Sq^{n-c}\Sq^c. \]
	Thus $\Sq^n$ can be written in terms of $\Sq^{2^k}$ and $\Sq^i$ with $i<n$. To avoid infinite descent, we see that $\Sq^n$ can be written using only $\Sq^{2^j}$ for various $j$.
\end{proof}



Given a sequence $I=(i_1,i_2,\ldots,i_r)$, we use the shorthand $\Sq^I=\Sq^{i_1}\Sq^{i_2}\cdots\Sq^{i_r}$. The sequence $I$ is called \textbf{admissible} if $i_j\ge 2i_{j+1}$ for all $j$; hence the monomials $\Sq^I$ are precisely those with no nontrivial Ad\'em relations. It follows from the Ad\'em relations that $\m A$ is generated as a vector space by the monomials $\Sq^I$ for all admissible sequences $I$. Serre established the linear independence of these classes, hence they form a basis known as the Serre-Cartan basis for $\m A$.
\\

The coproduct also has a particularly elegant description using the $\Sq^I$ notation: given any sequence $I=\left(i_1,\ldots,i_n\right)$, we have
\begin{equation}\label{nicecoproduct}
	\D\left(\Sq^I\right)=\sum_{I_1+I_2=I}\Sq^{I_1}\T\Sq^{I_2},
\end{equation}
where $I_1+I_2$ denotes component-wise addition. This formula becomes clear upon examination: applying $\D$ to each factor, the terms in $\D(\Sq^I)$ correspond to a choice of one term from each $\D(\Sq^{i_j})$. These correspond to a choice of number $a_j$ between $0$ and $i_j$ which indicates the term $\Sq^{a_j}\T\Sq^{b_j}$ in $\D(\Sq^{i_j}$ (where $a_j+b_j=i_j$). Hence the terms of $\Sq^I$ correspond to any and all choices of sequences which are component-wise between $(0,0,\ldots,0)$ and $(i_1,i_2,\ldots, i_n)$. Each such sequence of course has a complimentary sequence such that their sum is $I$.

\begin{remark}
In the coproduct formula (\ref{nicecoproduct}), it is worth emphasizing that none of $I$, $I_1$, or $I_2$ are required to be admissible. Also, we are \textit{adding} sequences together rather than concatenating: for example, $\D(\Sq^{3,5,2})$ will a priori have $72$ terms including $\Sq^{3,5,2}\T 1$, $\Sq^{1,3,2}\T\Sq^{2,2,0}$, and $\Sq^{1,1,1}\T\Sq^{2,4,1}$. Of course, many of these terms vanish due to nontrivial relations. The second and third terms mentioned vanish since $\Sq^{1,3,2}=0$ and $\Sq^{1,1,1}=0$.
\end{remark}

\section{A closer look at $\m A_*$}

One must be careful when taking the dual of an infinite dimensional space. Here, $\m A_*$ is the subspace of the $\Hom_{\Z_2}(\m A,\Z_2)$ generated by functionals with finite-dimensional support. Since $\m A$ is infinite dimensional, the true linear dual has strictly larger dimension, but we nonetheless refer to $\m A_*$ as simply ``dual'' to $\m A$ and vice versa. This is possible since, by only allowing functionals with finite dimensional support, we ensure that $\m A_*$ is isomorphic to $\m A$ as a graded vector space.
\\

Define $I_n=\left(2^{n-1},2^{n-2},\ldots, 4,2,1\right)$ for $n>0$ and $I_0=(0)$ (that is, $\Sq^{I_0}=\Sq^0=1$). For all $n\ge 0$, one can define $\xi_n$ as the linear dual to $\Sq^{I_n}$ with respect to the Serre-Cartan basis for $\m A$. That is, $\langle \xi_n,\Sq^I\rangle$ is $1$ if $I=I_n$ and is $0$ if $I$ is any other admissible sequence. Of course, this also defines $\langle \xi_n,\Sq^I\rangle$ for any $I$ after applying the Adem relations. We will write $\xi_0=1$. We noted earlier that cocommutativity of $\m A$ ensures that $\m A_*$ is commutative. Some relatively simple combinatorics will show that, for every $n$, the number of admissible sequences of degree $n$ corresponds bijectively to the number of partitions of $n$ using the numbers $2^k-1=\deg(I_k)$. This turns out to not give false hope: Serre showed that, as an algebra, $\m A_*$ is the polynomial ring $\Z_2[\xi_1,\xi_2,\ldots]$. Note that $\deg(\xi_n)=2^n-1$, and as with the notation $\Sq^I$, we write $\xi^J=\xi_1^{j_1}\xi_2^{j_2}\cdots\xi_r^{j_r}$ for any a sequence $J=(j_1,\ldots, j_r)$ of nonnegative integers. Note that using admissible sequences correspond to Serre-Cartan basis elements for $\m A$, but for $\m A_*$ it is more convenient to use arbitrary sequences (with finite length and nonnegative integer entries).
\\

 The comultiplication on $\m A_*$ is given by
\[\mu^*(\xi_n)=\sum_{i=0}^n\xi_{n-i}^{2^i}\T\xi_i=\xi_n\T 1+\xi_{n-1}^2\T \xi_1+\xi_{n-2}^4\T\xi_2+\cdots+\xi_1^{2^{n-1}}\T\xi_{n-1}+1\T\xi_n.\]
Sometimes it is convenient to use the Hopf conjugates $\zeta_n\coloneqq\chi(\xi_n)$. Then the coproduct is
\[\mu^*(\zeta_n)=\sum_{i=0}^n\zeta_i\T\zeta_{n-i}^{2^i}=1\T\zeta_n+\zeta_1\T\zeta_{n-1}^2+\zeta_2\T\zeta_{n-2}^4+\cdots+\zeta_{n-1}\T\zeta_1^{2^{n-1}}+\zeta_n\T 1.\]

In theory, one can compute the coproduct directly by dualizing from the product on $\m A$. Consider $\mu^*(\xi_2)$, for example. By definition, $\xi_2$ is dual to $\Sq^{2,1}$. There are only two vector space generators in degree $3$, and we have
\[\copo{\xi_2}{\Sq^{2,1}}=1\and\copo{\xi_2}{\Sq^3}=0. \]
The basis elements which may be summands of $\mu^*(\xi_2)$ are
\[\xi_2\T 1,\quad \xi_1^3\T1,\quad \xi_1^2\T\xi_1,\quad \xi_1\T\xi_1^2, \quad 1\T\xi_1^3, \quad\text{and}\quad 1\T\xi_2,\]
and we can simply check which of these are summands in by evaluating against basis elements in the Serre-Cartan basis elements $\Sq^3\T 1$, $\Sq^{2,1}\T\Sq^1$, etc. For example,
\[\copo{\mu^*\left(\xi_2\right)}{\Sq^3\T 1}=\copo{\xi_2}{\mu\left(\Sq^3\T 1\right)}=\copo{\xi_2}{\Sq^3}=0.\]
In this fashion we see that $\mu^*(\xi_2)$ evaluates to $1$ on precisely the generators $\Sq^{2,1}\T 1$, $\Sq^2\T\Sq^1$ and $1\T\Sq^{2,1}$. By examining the lower dimensions one finds $\mu^*(\xi_2)=\xi_2\T 1+ \xi_1^2\T\xi_1+1\T\xi_2$.
\\

We defined $\xi_n$ to be dual to $\Sq^n$ in the basis consisting of monomials $\Sq^I$ with $I$ admissible; now we set $\Sq(I)=\Sq(i_1,i_2,\ldots,i_n)$ to be dual to $\xi^I=\xi_1^{i_1}\xi_2^{i_2}\cdots\xi_n^{i_n}$ with respect to the basis consisting of all monomials $\xi^J$ for any sequence $J$. Note that the degree of $\xi^I$ is $\sum_j i_j(2^j-1)$, and we can obtain an element of this degree in the original basis of $\m A$ as follows: create an admissible sequence $I'$ such that the excess $i_j'-2i_{j+1}'$ in position $j$ is equal to $i_j$. One might hope that $\Sq^{I'}$ is equal to $\Sq(I)$, but this is not true in general. However, we can order these bases to obtain a bilinear pairing which recovers $\Sq^{I'}$ from $\Sq(I)$ and vice versa. This ordering on both the sets of admissible sequences and the sets of sequences is lexicographic from the right. For example,
	\[(100,21,3)<(42,22,3)<(16,8,4)<(8,4,2,1). \]

\section{Pairing between $\m A$ and $\m A_*$}

Let $\m I$ be the set of finite sequences of nonnegative integers and let $\m J\subset\m I$ be the set of such sequences which are admissible. We give $\m I$ and $\m J$ total orders using the lexicographic ordering from the right. That is, longer sequences have higher order, and, if two sequences have the same length, we compare the rightmost entries where they differ (larger means higher order). For example,
\[(5)<(3,1)<(4,1)<(0,2)<(0,0,1). \]
Note that this order applies to both $\m I$ and $\m J$. We also have an order-preserving bijection $\sigma\colon J\to I$ given by $\sigma(j_1,\ldots, j_n)=(j_1-2j_2,j_2-2j_3,\ldots, j_{n-1}-2j_n,j_n)$.
\begin{lemma}
	Let $I$ and $J$ be admissible sequences. If $I\ge J$, then
	\[\left\langle\xi^{\sigma(I)},\Sq^J\right\rangle=\begin{cases}
		1\pw I=J\\
		0\pw I>J.
	\end{cases} \]
\end{lemma}
\begin{proof}
Assuming $I$ is nontrivial, we can write $I=(i_1,\ldots, i_n)$ with $i_n>0$. Then we can subtract $1$ from the last entry of $\sigma(I)$ to get $\rew I\coloneqq(i_1-2i_2,i_2-2i_3,\ldots, i_{n-1}-2i_n,i_n-1)$, and we have
	\begin{align*}
		\left\langle \xi^{\sigma(I)},\Sq^J\right\rangle&=\left\langle\D^*\left(\xi^{\rew I}\T\xi_n\right),\Sq^J\right\rangle\\
		&=\left\langle \xi^{\rew I},\D\left(\Sq^J\right)\right\rangle\\
		&=\sum_{J_1+J_2=J}\left\langle\xi^{\rew I}\T\xi_n,\Sq^{J_1}\T\Sq^{J_2}\right\rangle.
	\end{align*}
	Since $\xi_n$ is dual to $I_n\coloneqq\left(2^{n-1},2^{n-2},\ldots,4,2,1\right)$, the only possible nonzero term here is when $J_2=I_n$. Thus
	\begin{align*}
		\sum_{J_1+J_2=J}\left\langle\xi^{\rew I}\T\xi_n,\Sq^{J_1}\T\Sq^{J_2}\right\rangle&=\left\langle\xi^{\rew I},\Sq^{J-I_n}\right\rangle \left\langle\xi_n,\Sq^{I_n}\right\rangle\\
		&=\left\langle\xi^{\rew I},\Sq^{J-I_n}\right\rangle.
	\end{align*}
	At this point, notice that $J$ must have at least length $n$ for $J-I_n$ to be nonnegative. In addition, the assumption $I\ge J$ means the length of $J$ is at most $n$, so we can write $J=(j_1,\ldots, j_n)$ (note that $I\ge J$ also implies $i_n\ge j_n$). If $I=J$, then we can repeat this process until reaching $\langle\xi_k,\Sq^{I_k}\rangle=1$ for some minimal $k$. We now assume $I>J$. It's easy to see that $J-I_n=(j_1-2^{n-1},j_2-2^{n-2},\ldots, j_{n-1}-2,j_n-1)$ is admissible, and since $i_n\ge j_n$, we can repeat this process $j_n-1$ more times. If $i_n=j_n$, we start from the beginning, replacing $\rew I$ with $(i_1-2i_2,\ldots, i_{n-1}-2i_n)$ and $J$ with $(j_1-2^{n-1}j_n,j_2-2^{n-2}j_n,\ldots, j_{n-2}-4j_n,j_{n-1}-2j_n)$. In this case, $i_{n-1}-2i_n\ge j_{n-1}-2j_n$, and we continue as long as possible. Eventually, we reach some value $\langle\xi^{I'},\Sq^{J'}\rangle$ where the length of $I'$ is strictly larger than the length of $J'$. As we saw, in this case $\langle\xi^{I'},\Sq^{J'}\rangle=0$, which completes the proof.
\end{proof}

\section{The Hopf algebra antipode map and Hopf subalgebras}

Since $\m A_*$ is also a Hopf algebra, it has an antipode map $\chi\colon\m A_*\to\m A_*$ characterized by the commutativity of the following diagram.

\begin{figure}[h!]\centering
	\begin{tikzcd}[column sep={1.5cm,between origins}, row sep={1.5cm,between origins}]
		&\m A_*\T\m A_* \ar[rr,"1\T\chi"] && \m A_*\T\m A_*\ar[dr] &\\
		\m A_*\ar[ur]\ar[rr]\ar[dr] && \Z_2\ar[rr] && \m A_*\\
		&\m A_*\T\m A_* \ar[rr,"\chi\T1"] && \m A_*\T\m A_*\ar[ur] &
	\end{tikzcd}
\end{figure}

\noindent The diagonal maps are the obvious product and coproduct maps, and the unlabeled horizontal maps are the augmentation (counit) and unit maps. For $n\ge 0$, the middle row takes $\xi_n$ to $0$. Comparison with the top row shows $\sum_{i+j=n}\xi_i^{2^j}\chi(\xi_j)=0$. In particular $\chi(\xi_n)$ is determined inductively via
\[\chi(\xi_n)=\xi_1^{2^{n-1}}\chi(\xi_{n-1})+\xi_2^{2^{n-2}}\chi(\xi_{n-2})+\cdots+\xi_{n-2}^4\chi(\xi_2)+\xi_{n-1}^2\chi(\xi_1)+\xi_n.\]
Adopting the semi-standard convention $\zeta_n\coloneqq\chi(\xi_n)$, we equivalently have
\[\zeta_n=\xi_n+\sum_{i=1}^{n-1}\xi_i^{2^{n-i}}\zeta_{n-i}=\xi_n+\xi_{n-1}^2\zeta_1+\xi_{n-2}^4\zeta_2+\cdots+\xi_1^{2^{n-1}}\zeta_{n-1}. \]

For $x\in \m A$, we write $\D(x)=\sum_i x'_i\T x''_i$. The fact that $\m A$ is connected means that, if $x\in\m A^+$, then the conjugate satisfies $\sum_i\chi(x'_i)x''_i=0=\sum_ix'_i \chi(x''_i)$, where $\m A^+$ denotes the positively graded part of $\m A$. In this case, $\D(x)$ is the sum of $1\T x+x\T 1$ plus terms in $\m A^+\T \m A^+$, hence one can inductively determine $\chi(x)$ by this property. For example, with the Steenrod algebra we have $\D(\Sq^1)=\Sq^1\T 1+1\T \Sq^1$, hence $\mu(\chi(\Sq^1)\T 1)+\mu(\chi(1)\T\Sq^1)=0$ and $\chi(\Sq^1)=\Sq^1$. Next $\D(\Sq^2)=\Sq^2\T 1+\Sq^1\T\Sq^1+1\T\Sq^2$ and so $\chi(\Sq^2)=\chi(\Sq^1)\Sq^1+\Sq^2=\Sq^2$. For $\Sq^3$ the antipode is nontrivial: $\chi(\Sq^3)=\chi(\Sq^2)\Sq^1+\chi(\Sq^1)\Sq^2+\Sq^3=\Sq^2\Sq^1$.
\\

We say $E\subseteq\m A$ is a sub-Hopf algebra of $\m A$ if $E$ is a Hopf algebra whose structure maps are restrictions of those for $\m A$. In particular, $E$ is a subalgebra of $\m A$, the dual $E_*$ is a sub-coalgebra of $\m A_*$, and the conjugation maps preserve $E$ and $E_*$ as subspaces.
\\

The correct notion of the quotient of $\m A$ by a Hopf subalgebra $E$ is a bit subtle: since $E$ is unital, the quotient by the left or right ideal generated by $E$ gives the zero ring. Instead, we use the semi-standard notation $\m A\hmod E$ to denote the Hopf algebra quotient $\Z_2\T_{E}\m A$. In light of this, the dual notion involves the cotensor product. That is, we say $\Z_2\T_E\m A$ is dual to $\Z_2\square_E\m A_*$.

\section{A brief review of comodule theory}

We now include a superficial review of comodules and cotensor products. Given an algebra $A$ over a base ring $k$, a right $A$-module $M$ has the structure of a map $\mu_M\colon M\T_k A\to M$; a left $A$-module has $\mu_N\colon R\T_k N\to R$. Completely dual to this, a right $A$-comodule $C$ has coaction map $\D_C\colon C\to C\T_k A$, and a left $A$-comodule $D$ similarly has $\D_D\colon D\to A\T_k D$. Now the tensor product $M\T_A N$ may be defined as the cokernel of
\[M\T A\T N\xrightarrow{\mu_M\T 1-1\T \mu_N} M\T N,\]
where the unadorned tensor $\T$ denotes $\T_k$. Dually, the cotensor product $C\square_R D$ is the kernel of
\[C\T D\xrightarrow{\D_C\T 1-1\T\D_D}C\T A\T D. \]

As is consistent with the literature, we will refer to elements $Q_i\in\m A$ defined as the linear dual to $\xi_i$ with respect to the monomial basis. The following facts can be found in \cite{rog} $15.5$.

\begin{lemma}\label{dualquotients}
	For all $n$, $\m A(n)\coloneqq\langle \Sq^1,\Sq^2,\Sq^4,\ldots, \Sq^{2^n}\rangle$ and $E(n)\coloneqq\langle Q_0,Q_1,\ldots, Q_n\rangle$ are sub-Hopf algebras of $\m A$. Moreover, $E(n)$ is the exterior algebra on $Q_0,\ldots, Q_n$. The quotient $\m A\hmod\m A(n)=\Z_2\T_{\m A(n)}\m A$ is dual to
	\[\Z_2\square_{\m A(n)_*}\m A_*=\Z_2\left[\xi_1^{2^{n+1}},\xi_2^{2^n},\ldots,\xi_n^4,\xi_{n+1}^2,\xi_{n+2},\xi_{n+3},\xi_{n+4},\cdots\right]. \]
	The quotient $\m A\hmod E(n)=\Z_2\T_{E(1)}\m A$ is dual to
	\[\Z_2\square_{E(n)_*}\m A_*=\Z_2[\xi_1^2,\xi_2^2,\ldots,\xi_n^2, \xi_{n+1}^2,\xi_{n+2},\xi_{n+3},\xi_{n+4},\ldots]. \]
\end{lemma}

This lemma is overpowered for our purposes, but it shows $\m A\hmod\m A(1)$ is dual to
\[\Z_2[\xi_1^4,\xi_2^2,\xi_3,\xi_4,\ldots]= \Z_2[\zeta_1^4,\zeta_2^2,\zeta_3,\zeta_4,\ldots],\]
and $\m A\hmod E(1)$ is dual to
\[\Z_2[\xi_1^2,\xi_2^2,\xi_3,\xi_4,\ldots]= \Z_2[\zeta_1^2,\zeta_2^2,\zeta_3,\zeta_4,\ldots].\]

The topological significance is easy to state: for $\ko$ the connective cover of the (real) $\KO$-theory spectrum (and $\ku$ the connective cover for complex $K$-theory),
\[H^*\ko=\m A\hmod \m A(1)\quad\text{and hence}\quad H_*\ku=\Z_2[\xi_1^4,\xi_2^2,\xi_3,\xi_4,\ldots].\] On the other hand,
\[H^*\ku=\m A\hmod E(1)\quad\text{and}\quad H_*\ku=\Z_2[\xi_1^2,\xi_2^2,\xi_3,\xi_4,\ldots].\]
Note that $E(0)=\m A(0)$ is generated by $1$ and $\Sq^1$, and we have also described $H^*H\Z=\m A\hmod\m A(0)$ and $H_*H\Z=\Z_2[\xi_1^2,\xi_2,\xi_3,\ldots]$.
\\

For further reference, see \cite{adams} (proposition $16.6$ of section III), \cite{stong} (page $330$), and \cite{abp} (page $287$).

\section{$\m A(1)$ modules and $H^*\BPSp(3)$}

Recall that $\m A(1)$ is the subalgebra of $\m A$ generated by $\Sq^1$ and $\Sq^2$. It is easy to verify that $\m A(1)$ has basis \[\{1,\Sq^1,\Sq^2,\Sq^3,\Sq^{2,1},\Sq^{3,1},\Sq^{4,1}+\Sq^5,\Sq^{5,1}\}.\]
Here we have written elements in terms of $\Sq^I$ with admissible $I$, but note that we have several nontrivial relations: for example,
\begin{align*}
	\Sq^{1,2}&=\Sq^3\\
	\Sq^{2,2}&=\Sq^{3,1}\\
	\Sq^{2,1,2}&=\Sq^{2,3}=\Sq^{4,1}+\Sq^5\\
	\Sq^{2,2,2}&=\Sq^{2,1,2,1}=\Sq^{1,2,1,2}=\Sq^{5,1}.
\end{align*}

We can depict $\m A(1)$-modules as graphs with a node for each basis element, a short edge indicating left multiplication by $\Sq^1$, and a long edge indicating left multiplication by $\Sq^2$. Diagrammatically, we indicate degree with vertical position. In addition to the regular $\m A(1)$-module, we consider $\Z_2$ as a module concentrated in degree zero. Write $I$ for the augmentation ideal of $\m A(1)$. The inclusion $I\hookrightarrow\m A(1)$ and augmentation $\m A(1)\to\Z_2$ maps are of course maps of $\m A(1)$-modules. The other useful $\m A(1)$-modules are the ``joker'' $J\coloneqq \m A(1)/\m A(1)(\Sq^3)$ and the module $K\coloneqq \m A(1)/\m A(1)(\Sq^1,\Sq^2\Sq^3)$. These submodules are depicted in figure \ref{submodulesA}.


\newcommand{\xmin}{0}
\newcommand{\xmax}{13}
\newcommand{\ymin}{-1}
\newcommand{\ymax}{6}

\begin{figure}[h!]\centering
\begin{tikzpicture}
	\foreach \i in {\xmin,...,\xmax}{
	\draw [thin, gray] (\i-0.5,\ymin-0.5) -- (\i-0.5,\ymax+0.5);
	\draw node at (\i,\ymin-1) {$\i$};}
	\draw [thin, gray] (\xmax+0.5,\ymin-0.5) -- (\xmax+0.5,\ymax+0.5);
	
	\foreach \i in {\ymin,...,\ymax}{
	\draw [thin, gray] (\xmin-0.5,\i-0.5) -- (\xmax+0.5,\i-0.5);
	\draw node at (\xmin-1,\i) {$\i$};}
	\draw [thin, gray] (\xmin-0.5,\ymax+0.5) -- (\xmax+0.5,\ymax+0.5);
	
	\diagramAone{1}{0}
	\diagramI{4}{1}
	\diagramJ{8}{2}
	\diagramK{10}{0}
	\diagramZ{12}{0}

	\node at (2,-1) {$\m A(1)$};
	\node at (5,-1) {$I$};
	\node at (8,-1) {$\Sigma^2J$};
	\node at (10,-1) {$K$};
	\node at (12,-1) {$\Z_2$};
\end{tikzpicture}
\caption{Relevant submodules of $\m A(1)$.}
\label{submodulesA}
\end{figure}

A key computation in \cite{stolz} was the the $\m A(1)$-module structure of $H^*\BPSp(3)$. We revisit Stolz's computation beginning with the $\m A(1)$ action on $H^*\BPSp(3)$ (\cite{stolz}). We will rely  on Kono's determination in \cite{kono} that
\[H^*\BPSp(3)=\Z_2[t_2,t_3,t_8,t_{12}],\]
where $\deg t_i=i$. To compute the the actions of $\Sq^1$ and $\Sq^2$ on generators, Stolz used representations of the subgroup $\op{P}(\Sp(1)^3)\subset\PSp(3)$ using the geometric fact that $\Sp(1)$ naturally acts on $\H$ as the unit quaternions. He found that $\Sq^1$ and $\Sq^2$ act as follows:
\begin{align*}
	&\Sq^1(t_2)=t_3& &\Sq^1(t_3)=0& &\Sq^1(t_8)=0& &\Sq^1(t_{12})=0&\\
	&\Sq^2(t_2)=t_2^2& &\Sq^2(t_3)=t_2t_3& &\Sq^2(t_8)=0& &\Sq^2(t_{12})=t_2t_{12}.&
\end{align*}

\begin{lemma}\label{firstsplit}
	As $\m A(1)$-modules, $H^*\BPSp(3)\cong \Z_2[t_8,t_{12}^2]\T\left(\Z_2[t_2,t_3]\oplus \Z_2[t_2t_{12},t_3t_{12}]\right)$.
\end{lemma}
\begin{proof}
	The Cartan formula shows
	\[\Sq^1(xy)=\Sq^1(x)y+x\Sq^1(y)\and \Sq^2(xy)=\Sq^2(x)y+\Sq^1(x)\Sq^1(y)+x\Sq^2(y),\]
	hence to verify that $\Z_2[t_8,t_{12}^2]$ is an $\m A(1)$-submodule, it suffices to check that generators $\Sq^1$ and $\Sq^2$ take generators $t_8$ and $t_{12}^2$ to $\Z_2[t_8,t_{12}^2]$. This is immediate by Stolz's computation. The same is true for $\Z_2[t_2,t_3]$. Finally for $\Z_2[t_2t_{12},t_3t_{12}]$ we check
	\begin{align*}
		\Sq^1(t_2t_{12})&=\Sq^1(t_2)t_{12}+t_2\Sq^1(t_{12})=t_2t_{12}\\
		\Sq^2(t_2t_{12})&=t_2^2t_{12}+t_2^2t_{12}=0\\
		\Sq^1(t_3t_{12})&=0\\
		\Sq^2(t_3t_{12})&=t_2t_3t_{12}+t_2t_3t_{12}=0.
	\end{align*}
	Clearly we have an injection $\Z_2[t_8,t_{12}^2]\T\left(\Z_2[t_2,t_3]\oplus \Z_2[t_2,t_3]t_{12}\right)\to\Z_2[t_2,t_3,t_8,t_{12}]$. Since these modules are finite in each degree we can verify that all elements have been accounted for with generating functions. A generating function for $\Z_2[t_2,t_3,t_8,t_{12}]$ is
	\[\frac{1}{(1-x^2)(1-x^3)(1-x^8)(1-x^{12})}\].
	On the other hand, for $\Z_2[t_8,t_{12}^2]\T\left(\Z_2[t_2,t_3]\oplus \Z_2[t_2,t_3]t_{12}\right)$ we have the generating function
	\begin{align*}
	\frac{1}{(1-x^8)(1-x^{24})}&\left(\frac{1}{(1-x^2)(1-x^3)}+\frac{x^{12}}{(1-x^2)(1-x^3)}\right)\\ \\
	&\qquad=\frac{1}{(1-t^8)(1-t^{12})(1+t^{12})}\left(\frac{1+x^{12}}{(1-x^2)(1-x^3)}\right)\\ \\
	&\qquad=\frac{1}{(1-x^2)(1-x^3)(1-x^8)(1-x^{12})}
	\end{align*}
\end{proof}

Next we aim to write $\Z_2[t_8,t_{12}^2]$, $\Z_2[t_2,t_3]$ and $\Z_2[t_2t_{12},t_3t_{12}]$ in of $\m A(1)$, $\Z_2$, $I$, $J$, and $K$. The utility of \ref{firstsplit} is that $\Z_2[t_8,t_{12}]$ has trivial action by $\Sq^1$ and $\Sq^2$, hence is the direct sum of $\Z_2$ with a trivial module. This reduces our task to finding the $\m A(1)$-module structure of $\Z_2[t_2,t_3]$ and $\Z_2[t_2t_{12},t_3t_{12}]$. We refer the reader to \cite[proposition $6.5$]{stolz} for the remainder of the proof.


\begin{theorem}[Stolz lemma $7.6$]\label{2cond}
	Suppose $f\colon C\to D$ is a map of left $\m A(1)$-modules, where $C$ is a direct sum of suspensions of $\Z_2$, $J$, and $\m A(1)$, and $D$ is a direct sum of suspensions of $\Z_2$, $I$, $J$, $K$, and $\m A(1)$. If $f$ is injective and $f$ induces an injection $H(C;Q_0)\to H(D;Q_0)$, then $f$ is a split injection.
\end{theorem}

\section{$E(1)$ modules and $H^*\BSU(3)$}

Let $E(1)$ be the subalgebra of $\m A$ generated by $Q_0=\Sq^1$ and $Q_1=\Sq^2\Sq^1+\Sq^3$. Note that $E(1)$ is a subalgebra of $\m A(1)$, and the inclusion $E(1)\to\m A(1)$ yields a restriction of scalars functor $\m A(1)\Lmod\Rightarrow E(1)\Lmod$ right adjoint to induction $\m A(1)\T_{E(1)}-\colon E(1)\Lmod\Rightarrow\m A(1)\Lmod$. As before we consider $\Z_2$ as a submodules of $E(1)$ concentrated in degree $0$. Let $L$ be the augmentation ideal of $E(1)$ and let $C=E(1)/E(1)\Sq^1$. These $E(1)$-modules are shown in figure \ref{submodulesE}

\newcommand{\xmina}{0}
\newcommand{\xmaxa}{8}
\newcommand{\ymina}{-1}
\newcommand{\ymaxa}{4}

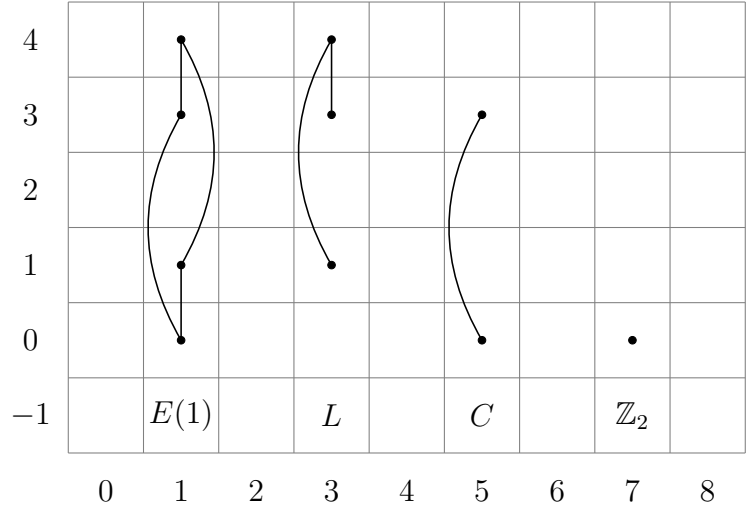
\begin{figure}[h!]\centering
\begin{tikzpicture}
	\foreach \i in {\xmina,...,\xmaxa}{
	\draw [thin, gray] (\i-0.5,\ymina-0.5) -- (\i-0.5,\ymaxa+0.5);
	\draw node at (\i,\ymina-1) {$\i$};}
	\draw [thin, gray] (\xmaxa+0.5,\ymina-0.5) -- (\xmaxa+0.5,\ymaxa+0.5);
	
	\foreach \i in {\ymina,...,\ymaxa}{
	\draw [thin, gray] (\xmina-0.5,\i-0.5) -- (\xmaxa+0.5,\i-0.5);
	\draw node at (\xmina-1,\i) {$\i$};}
	\draw [thin, gray] (\xmina-0.5,\ymaxa+0.5) -- (\xmaxa+0.5,\ymaxa+0.5);
	
	\diagramEone{1}{0}
	\diagramL{3}{0}
	\diagramC{5}{0}
	\diagramZ{7}{0}
	
	\node at (1,-1) {$E(1)$};
	\node at (3,-1) {$L$};
	\node at (5,-1) {$C$};
	\node at (7,-1) {$\Z_2$};
\end{tikzpicture}
\caption{Relevant submodules of $E(1)$.}
\label{submodulesE}
\end{figure}

\begin{lemma}
We have isomorphisms of $E(1)$-modules
\begin{align*}
	\m A(1)&\cong E(1)\oplus\Sigma^2 E(1)\\
	I&\cong L\oplus\Sigma^2E(1)\\
	J&\cong\Z_2\oplus\Sigma^{-2}E(1)\\
	K&\cong\Sigma\inv L.
\end{align*}
\end{lemma}

\begin{lemma}[Stolz proposition $6.5$]\label{bpsp}
	As an $\m A(1)$-module, $H^*\BPSp(3)$ is the direct sum of a free module with a direct sum of $8i$-fold suspensions of $\Z_2$, $\Sigma^{-1}I$, $\Sigma^4J$, and $\Sigma^4K$.
\end{lemma}

\begin{corollary}
	As an $E(1)$-module, $H^*\BPSp(3)$ is a direct sum of $4i$-fold suspensions of $\Z_2$ and $\Sigma^{-1}L$.
\end{corollary}
\begin{proof}
	Thus lemma \ref{bpsp} shows that $H^*\BPSp(3)$ is, as an $E(1)$-module, the direct sum of a free module with a direct sum of $8i$-fold suspensions of $\Z_2$, $\Sigma^{-1}L$, $\Sigma^4\Z_2$, and $\Sigma^3L$.
\end{proof}

As an algebra, $H^*\BSU(3)$ is the polynomial ring $\Z_2[y_4,y_6]$, where $\deg(y_i)=i$. In theorem \ref{sqc} we compute the actions by $\Sq^1$ and $\Sq^2$, and we summarize the result here: $\Sq^1(y_6)=\Sq^2(y_6)=\Sq^1(y_4)=0$ and $\Sq^2(y_4)=y_6$. Hence $\Z_2[y_4^2,y_6]$ is an $E(1)$-submodule of $H^*\BSU(3)$ with trivial action. We can now make a better description of $H^*\BSU(3)$.

\begin{theorem}
	As an $E(1)$-module, $H^*\BSU(3)\cong \Z_2[y_4^2]\oplus C\T_{\Z_2}\Z_2[y_4^2,y_6]y_4$. Here $\Z_2[y_4^2]$ and $\Z_2[y_4^2,y_6]y_4$ are trivial $E(1)$-modules which may be considered as submodules of $\Z_2[y_4,y_6]\cong H^*\BSU(3)$.
\end{theorem}
\begin{proof}
	Knowing the actions of $\Sq^1$ and $\Sq^2$ on generators $y_4$ and $y_6$, it is easy to verify $\Z_2[y_4^2]$ is a trivial submodule of $\Z_2[y_4,y_6]$. As vector spaces,
	\[\Z_2[y_4,y_6]\cong\Z_2[y_4^2]\oplus\Z_2[y_4^2,y_6]y_4\T C.\]
	To see this, observe that a generating function in $t$ for the right hand side is
	\begin{align*}
	\frac{1}{1-t^8}+\frac{t^4(1+t^2)}{(1-t^8)(1-t^6)}&=\frac{1-t^6+t^4(1+t^2)}{(1-t^4)(1+t^4)(1-t^6)}\\
	&=\frac{1}{(1-t^4)(1-t^6)}.
	\end{align*}
	To make this explicit, momentarily write $\{1,x\}$ for a $\Z_2$-basis of $C$ and define a map
	\[\p\colon\Z_2[y_4^2]\oplus\Z_2[y_4^2,y_6]y_4\T C\to\Z_2[y_4,y_6] \]
	by extending linearly from $\p(\theta_1,\theta_2\T 1)=\theta_1+\theta_2$ and $\p(\theta_1,\theta_2\T x)=\theta_1+\theta_2y_4\inv y_6$. Since the domain and codomain of $\p$ are finite dimensional in each grading, $\p$ is a linear isomorphism provided $\p$ is surjective. Elements of the form $y_4^{2m}y_6^n$ are obtained as $\p(0,y_4^{2m}y_6^{n-1}y_4\T x)$ if $n\ge 1$ and $\p(y_4^{2m},0)$ if $n=0$. For odd powers of $y_4$, observe $y_4^{2m+1}y_6^n=\p(0,y_4^{2m+1}y_6^n\T1)$. Hence $\p$ is an isomorphism of graded vector spaces
\end{proof}

\section{$\ko$ and $\ku$}

Let $\ko$ be the connective cover of $\KO$, meaning there is a map $\ko\to \KO$ inducing isomorphisms $\pi_n\ko\to\pi_n\KO$ for $n\ge 0$ and where $\pi_n\ko=0$ for $n<0$. Similarly one may define $\ku$ as the connective cover of the complex K-theory spectrum $\K$. Adams notes that one may take the zeroth space $\ku$ to be $\Z\times\BU$, hence $\ku^0(X)=\K^0(X)$ for any CW complex space $X$; in general the higher cohomology groups may differ. The similar statements apply for $\ku$. The $\A$ and $\AC$ genus correspond to spectrum maps $\MSpin\to\ko$ and $\MSpinc\to\ku$, and for this reason we care about the $\m A$-module structures of $\ko$ and $\ku$.

As an $\m A$-module, we can identify $H^*\ko=\m A\hmod\m A(1)$ and $H^*\ku=\m A\hmod E(1)$. Thus $H_*\ko$ and $H_*\ku$ are as given in lemma \ref{dualquotients}. We prove the result for $\ku$ here, following a proof by Bruner (\cite{bruner}) which in turn was adapted from Adams: for further reference, see proposition $16.6$ of section III of \cite{adams}.

\begin{lemma}
	As an $\m A$-module, $H^*\ku\cong\m A\hmod E(1)$.
\end{lemma}

\begin{proof}
We begin with the fibration induced by $\ku\to H\Z$. The total space is $\Sigma^2\ku$, giving a cofibration which extends to the right as in $\Sigma^3\xrightarrow{i}\ku\xrightarrow{j}H\Z\xrightarrow{k}\Sigma^3\ku$. We compare the two fiber sequences below.

\begin{figure}[h!]\centering
	\begin{tikzcd}
		& \Sigma^2\ku\ar[r] & \ast\ar[d]\\
		\Sigma^2\ku\ar[r,"i" above]\ar[ur,dashed] & \ku\ar[d,"j" left]\ar[ur,dashed] & \Sigma^3\ku\\
		& H\Z\ar[ur,"k" below right,dashed] &  
	\end{tikzcd}
\end{figure}

Write $\omega_n\in H^n\Sigma^n\ku$ for the fundamental cohomology class. The transgression in the upper right fibration is $\tau(\omega_2)=\omega_3$. It follows that $k^*\omega_3=\tau\omega_2$. We compute $\tau(\omega_2)$ by looking at the third space in each spectrum. Recall that
\[\ku=(\BU\times\Z,\U,\BU,\SU,\BSU,\SU\langle5\rangle,\SU\langle6\rangle,\ldots)\]
and
\[H\Z=(\Z,S^1,\CP^\infty,K(\Z,3),K(\Z,4),\ldots).\]
Restricting to the third spaces gives $\SU\langle5\rangle\to\SU\to K(\Z,3)$. Write $\I_3\in H^3K(\Z,3)$ for the fundamental class. We also know that $H^*\SU=\W(e_3,e_5,e_7,\ldots)$. Since $\SU\langle5\rangle$ is $4$-connected, the first nonzero cohomology group $H^5\SU\langle5\rangle=H^5\SU$ is generated by $e_5$. We also know $H^*K(\Z,3)$ (for example using $\CP^\infty\to \ast\to K(\Z,3)$): the first several generators are $1,\I_3,\Sq^2\I_3$, and $\Sq^3\I_3=\I_3^2$. Since $H^5\SU$ is one-dimensional (and generated by $e_5$), we necessarily have the transgression $\tau(e_5)=\Sq^3\I_3$. Now we turn to the cofibration $\ku\to H\Z\to \Sigma^3\ku$. Write $u\in H^0H\Z$ for the fundamental class. We have an isomorphism $\A\colon\m A/\m A\Sq^1\to H^*H\Z$ given by $\A(\theta)=\theta u$ and a diagram
\begin{figure}[h!]\centering
	\begin{tikzcd}
		& & \m A/\m A\Sq^1\ar[d,"\A"] & &\\
		0 \ar[r] & H^*\Sigma^3\ku\ar[r] & H^*(H\Z)\ar[r,"j^*"] & H^*\ku\ar[r,"i^*"] & 0
	\end{tikzcd}
\end{figure}
In addition, we have the exact sequence
\[0\to \m A/(\Sq^1,\Sq^3)\xrightarrow{f} \m A/(\Sq^1)\xrightarrow{g} \m A/(\Sq^1,\Sq^3)\to 0 \]
where $f(x)=x\Sq^3$ and $g$ is the quotient: note that $f$ is well defined since $\Sq^3\Sq^3=0$ and $\Sq^1\Sq^3=0$, and obviously $\im f\subseteq \ker g$. On the other hand $\ker g$ is generated by elements of the form $x\Sq^1+y\Sq^3$, where $x,y\in \m A$. Since $f(y)=y\Sq^3=x\Sq^1+y\Sq^3\in\m A/(\Sq^1)$, the sequence is exact. Our first computation shows that $k^*$ hits $\Sq^3$, so $j^*\A$ induces the map $\B$. Together we get the following diagram.

\begin{figure}[h!]\centering
\begin{tiny}
	\begin{tikzcd}[column sep=15pt]
		(\m A/\!(\Sq^1,\Sq^3))_{n-3}\ar[r,"f"]\ar[d,"\B"]&
		(\m A/\!(\Sq^1))_n\ar[r,"g"]\ar[d,"\A"]&
		\ar[d](\m A/\!(\Sq^1,\Sq^3))_n\ar[r,"\delta"]\ar[d,"\B"]&
		(\m A/\!(\Sq^1,\Sq^3))_{n-2}\ar[r,"f"]\ar[d,"\B"]&
		(\m A/\!(\Sq^1))_{n+1}\ar[d,"\A"]\\
		H^{n-3}\ku\ar[r,"k^*"] & H^n(H\Z)\ar[r,"j^*"] & H^n\ku\ar[r,"\delta"] & H^{n-2}\ku\ar[r,"k^*"] & H^{n+1}(H\Z)
	\end{tikzcd}
\end{tiny}
\end{figure}

We know that $\A$ is an isomorphism, and the Hurewicz theorem shows that the map $\B\colon\m A/\!(\Sq^1,\Sq^3)_n\to H^n\ku$ is an isomorphism for $n\le 2$. If $\B$ is an isomorphism for $n\le r$ with $r>2$, it follows from the five-lemma that the middle vertical map in the diagram is an isomorphism. By induction the result follows.
\end{proof}

\section{Extended modules and indecomposable quotients}

As a general fact, if $A$ is an algebra over a ring $B$ and $N$ is a left $B$-module, $A\T_B N$ is a left $A$-module in the obvious way. Further, if $N$ is free as an $B$-module, then  $A\T_B N$ is free as an $A$-module (the same is true for projective modules). In this context, we call $A\T_B N$ an extended $B$-module. In practice, $B$ may be easier to work with than $A$, so it is a desirable property for an $A$-module to be an extended $B$-module.
\\

In addition to free modules, extended modules arise in the more interesting case that $A$ is a Hopf algebra and $B$ is a sub-Hopf algebra of $A$. Two other notable extended modules are $H^*\ko=\m A\T_{\m A(1)}\Z_2$ and $H^*\ku=\m A\T_{E(1)}\Z_2$.
\\

Let $A$ and $B$ be a sub-Hopf algebras of $\m A$ with $B\subseteq A$, and let $N$ be a left $B$-module. The tensor product $A\T_B N$ is naturally a left $A$-module in the obvious way. If $B=\Z_2$, then $A\T_B N$ is a free $A$-module: generators of $N$ over $\Z_2$ each generate a copy of $A$ in $A\T_B N$. This is true more generally: if $N$ is a free as a left $B$-module, then $\m A\T_B N$ is a free left $\m A$-module. Lemma \ref{rspin} is essentially an observation which appears in \cite{stolz}.

\begin{lemma}\label{rspin}
	The map $D\colon \MSpin\to\ko$ induces a split surjection on homology which allows us to consider $H_*\ko$ as a subalgebra of $H_*\MSpin$.
\end{lemma}

\begin{lemma}(\cite[corollary $5.5$]{stolz}) \label{123} Let $Y$ be an $\MSpin$-module spectrum whose homology is bounded below and locally finite. Identify $H_*\ko$ as a subalgebra of $H_*\MSpin$ as in lemma \ref{rspin}. Then there is a functorial isomorphism
	\[H_*Y\to \m A_*\square_{\m A(1)_*}\ol{H_*Y}, \]
	where $\ol{H_*Y}\coloneqq \Z_2\T_{H_*\ko}H_*Y$. Given a map $f\colon X\to Y$ of $\MSpin$-module spectra, we can identify $f_*\colon H_*X\to H_*Y$ with $1\T\ol f_*$, where $\ol f_*$ is a map $\ol{H_*X}\to\ol{H_*Y}$. In addition, if $\ol f_*$ is a split surjection of $\m A(1)_*$-comodules, then $f_*$ is a split surjection of $\m A_*$-comodules.
\end{lemma}

\begin{lemma}\label{reducedc}
	Let $Y$ be an $\MSpinc$-module spectrum whose homology is bounded below and locally finite, and let $R\subset H_*\MSpinc$ be the subalgebra which $D_*$ maps isomorphically to $H_*\ku$. Let $\ol{H_*Y}\coloneqq \Z_2\T_RH_Y$ be the $R$-indecomposable quotient of $H_*Y$. Then there is a functorial isomorphism of $\m A_*$-comodules
	\[H_*Y\to \m A_*\square_{E(1)_*}\ol{H_*Y}. \]
	Given a map $f\colon X\to Y$ of $\MSpinc$-module spectra, we can identify $f_*\colon H_*X\to H_*Y$ with $1\T\ol f_*$, where $\ol f_*$ is a map $\ol{H_*X}\to\ol{H_*Y}$. In addition, if $\ol f_*$ is a split surjection of $E(1)_*$-comodules, then $f_*$ is a split surjection of $\m A_*$-comodules.
\end{lemma}

Lemmas \ref{123} and \ref{reducedc} are both particular cases of \cite[proposition $5.4$]{stolz}, so we defer to his proof.

\chapter{Transfer Maps}

Unless specified otherwise, we will assume that all spaces have a cell structure with finitely many cells in each dimension. For two spaces $X$, $Y$, we write $[X,Y]$ for homotopy classes of maps $X\to Y$. Given based spaces $(X,x_0)$ and $(Y,y_0)$, we similarly write $[(X,x_0),(Y,y_0)]$ for (basepoint preserving) homotopy classes of maps $(X,x_0)\to(Y,y_0)$. The notation $X_+$ means the disjoint union of $X$ with a point.

\section{Virtual vector bundles}

Define a virtual vector bundle over $X$ to be a pair $(X,f)$, where $f\colon X\to\BO\times\Z$ is an unbased map. For convenience, we sometimes omit the word ``vector'' and say $(X,f)$ is a ``virtual bundle.'' Homotopy classes of maps $X\to\BO\times\Z$ form the (unreduced) real $\K$-theory of $X$; i.e., $\KO(X)\coloneqq[X,\BO\times\Z]$. (Note that $\KO(X)$ can equivalently be defined as \textit{based} homotopy classes of maps $X_+\to\BO\times\Z$.) Given a virtual bundle $\xi=(X,f)$, we call $f$ the classifying map of $\xi$. Projection $\BO\times\Z\to\Z$ defines the rank of $\xi$. When $X$ is connected, the rank may be considered an integer $\op{rk}(\xi)\in\Z$. Two virtual bundles $(X,f)$ and $(X,g)$ are equivalent if they correspond to the same class of $\KO(X)$ (i.e., if $f$ and $g$ are homotopic).

By contrast, a genuine vector bundle $\eta$ over a connected space $X$ is classified by a map $X\to\BO(n)$, where $n$ is the rank of the bundle. Now $\eta$ yields a virtual bundle of rank $n$ via the composition
\[X\to\BO(n)\to\BO\times\{n\}\hookrightarrow\BO\times\Z.\]
Generally we abuse notation and write $\eta$ for both the vector bundle and virtual bundle. However, it is important to note that a virtual bundle does not always arise from a genuine vector bundle in this way. For example, virtual bundles can have negative rank, and, if $X$ is not compact, a map $X\to\BO\times\Z$ may not factor through $\BO(n)\times\Z$ for any $n$.

While there are virtual bundles which do not arise from vector bundles, there is also some loss of information in considering vector bundles as virtual bundles. For example, a stably trivial vector bundle yields a trivial virtual bundle of the same rank: suppose $\xi\oplus\bd{k}\cong\bd{n}\oplus\bd{k}$, where $n$ is the rank of $\xi$ and $\bd k$ is the trivial vector bundle of rank $k$. Then $X\xrightarrow{f}\BO(n)\hookrightarrow{\BO(n+k)}$ is nullhomotopic, so the resulting composition $X\to\BO$ is also nullhomotopic. This means that $\xi$, considered as a virtual bundle, is equivalent to the trivial virtual bundle of rank $n$.

\begin{remark}
It is immediate from the definitions that trivial virtual bundles are equivalent if and only if they have the same rank. It follows that vector bundles $\xi$ and $\eta$ over a connected base space $X$ yield equivalent virtual bundles if and only if $\op{rk}(\xi)=\op{rk}(\eta)$ and $\xi\oplus\bd{k}\cong\eta\oplus\bd{k}$ for some $k$.
\end{remark}

Suppose we have a genuine vector bundle $\xi$ of rank $n$ over a CW complex $X$. In particular, $X$ is paracompact, so we can embed $\xi$ into a trivial vector bundle $\bd{k}$ for some $k$. The complementary bundle $\xi^\perp$ (which can be defined topologically as the quotient bundle of $\ol{k}$ by $\xi$) can be used to form an additive inverse of $\xi$ as classes in $\KO(X)$. That is, $\KO(X)$ is an additive group, and, since $\xi+\xi^\perp=k$, we have $-\xi=\xi^\perp-k$. Hence $-\xi$ can be considered as the formal difference $\xi^\perp-k$ (which is a virtual bundle of rank $-n$). The (virtual) bundle $-\xi$ is called the \textbf{stable complement} of $\xi$. In particular, when $X$ is a manifold, we define the \textbf{stable normal bundle} of $X$ as the stable complement to the tangent bundle (considered as a virtual bundle).


Now suppose $\xi$ is stably trivial. Without loss of generality, $\xi\oplus\bd 1\cong\bd{n+1}$. Then $X\xrightarrow{f}\BO(n)\hookrightarrow{\BO(n+1)}$ is nullhomotopic, so $\rew f\colon X_+\to\BO\times\Z$ is homotopic to the constant map with image $\{b_0\}\times\{n\}$ and $\rew\xi\equiv\rew{\bd n}$. This shows that \textit{stably} trivial vector bundles become trivial virtual bundles of the same rank. Similar logic shows that \textit{stably equivalent} bundles of equal rank represent equivalent virtual bundles (of the same rank). It is important to not overlook the requirement that the genuine bundles had the same rank to begin with. For example, trivial virtual bundles of distinct ranks are \textit{not} equivalent.


\section{Thom spectra and Dold's theorem}

Given a CW spectrum $X$, we can suspend or desuspend $X$ to obtain $\Sigma^rX$ for any $r\in\Z$, where $(\Sigma^rX)_i=(X)_{r+i}$. A map of spectra $X\to Y$ of degree $d$ is defined as a map $\Sigma^dX\to Y$ of degree $0$ (a map $S^0\to S^0$ of degree $d$, for example, is a stable map $S^d\to S^0$). To avoid mentions of degree, we will instead use suspensions and desuspensions as appropriate to make all maps have degree $0$. For this reason, assume maps of spectra have degree $0$ by default. We will now describe some relevant spectra.

\begin{example}[Spectrum of a CW space]
  Any CW complex $X$ can be considered a (CW) spectrum whose $i$-th space $(X)_i$ is the basepoint $\{x_0\}$ for $i<0$ and $(X)_i=\Sigma^iX$ for $i\ge 0$. We use $X$ to denote the spectrum as well as the space.
\end{example}

\begin{example}[$\MO$ and similar Thom spectra]
  Write $\gamma_n$ for the universal $n$-plane bundle over $\BO(n)$. The inclusion $\BO(n)\to\BO(n+1)$ is induced by the composition $\O(n)\hookrightarrow\O(n)\times\O(1)\to\O(n+1)$, so $\gamma_{n+1}|_{\BO(n)}\cong\gamma_n\oplus 1$. The classical construction of Thom spaces yields spaces $\MO(n)$ and maps $\Sigma\MO(n)\to\MO(n+1)$ which give a spectrum $\MO$. Similar constructions give $\MSO$, $\MU$, $\MSpin$, $\MSpinc$, etc.
\end{example}

\begin{example}[Thom spectrum of a genuine vector bundle]\label{construction1}
  Suppose we have a (genuine) vector bundle $\xi$ over a CW complex $X$ of rank $k$. The classifying map $c_\xi\colon X\to\BO(k)$ induces $X^\xi\to\MO(k)$ (here $X^\xi$ denotes the classical Thom space rather than a spectrum; existence of $c_\xi$ is detailed in \cite{switzer} $11.33$). The pullback of $\gamma_{k+1}$ by the composition $X\to\BO(k)\to\BO(k+1)$ yields a map $\Sigma X^\xi\cong X^{\xi+1}\to\MO(k+1)$, and inductively we get a spectrum map $\Sigma^{-k}X^\xi\to\MO$. (Equivalently, we have a map $X^\xi\to\MO$ of degree $-k$.)
\end{example}

\begin{example}[Thom spectrum of a virtual bundle represented by a genuine bundle]\label{construction2}
  Suppose now that $\xi$ is a virtual bundle of (virtual) rank $k$ over a CW complex $X$ such that the classifying map $X\to\BO$ factors through $\BO(n)$ for some $n$. This is true, for example, when $X$ is finite dimensional (\cite{switzer} $6.35$). Let $\A$ be the induced vector bundle of rank $n$. Our first construction yields a classifying map $\Sigma^{-n}X^\A\to\MO$ of degree $0$. Define $X^\xi\coloneqq\Sigma^{k-n}X^\A$ so as to have the classifying map $\Sigma^{-k}X^\xi\to\MO$ (just as we did when $\xi$ was a genuine bundle of rank $k$).

  We should ensure that, in an appropriate sense, the spectrum $X^\xi$ and classifying map $\Sigma^{-k}X^\xi\to\MO$ are well-defined regardless of the choice of representative $\A$. Suppose a (genuine) $m$-bundle $\B$ also represents $\eta$. The maps $c_\A,c_\B\colon X\to\BO$ are then homotopic via a cellular homotopy $H$: to see this, apply relative cellular approximation to $(X\times I,X\times\{0,1\})$. This means $H$ factors through $\BO(\ell)$ for some $\ell$; in particular, $\A\oplus\bd{s}$ is isomorphic (as a vector bundle) to $\B\oplus\bd{t}$ for some $s,t\in\N$. Hence
  \[\Sigma^{k-n}X^\A\simeq \Sigma^{k-n-s}X^{\A\oplus\bd{s}}\simeq\Sigma^{k-n-s}X^{\B\oplus\bd{t}}\simeq\Sigma^{k-n-s+t}X^\B.\]
  Since $\A\oplus\bd{s}\cong\B\oplus\bd{t}$, we have $n+s=m+t$, so $\Sigma^{k-n-s+t}X^\B=\Sigma^{k-m}X^\B$. All equivalences shown are natural, so we have a natural equivalence $\Sigma^{k-n}X^\A\simeq\Sigma^{k-m}X^\B$ which, after desuspension, also carries the classifying map $\Sigma^{-n}X^\A\to\MO$ to the classifying map $\Sigma^{-m}X^\B\to\MO$.
\end{example}

Starting with a genuine bundle $\xi$, we can consider $\xi$ to be virtual, in which case Example \ref{construction1} and Example \ref{construction2} give potentially distinct definitions of $X^\xi$. However, it is trivial to check that both constructions agree.

\begin{example}[Thom spectrum of a negative bundle]
Let $\xi$ be a genuine vector bundle of rank $k$ over $X$. As a virtual bundle, $\xi$ has additive inverse $-\xi$ of rank $-k$. Explicitly, since $X$ is paracompact, we can include $\xi$ into a trivial vector bundle of rank $n$. The orthogonal complement $\xi^\perp$ is then a genuine vector bundle of rank $n-k$. We simply let $-\xi$ be the virtual bundle represented by $\xi^\perp$ and with rank $-k$ (recall that the rank is just given by $\BO\times\Z\to\Z$). Alternatively, and perhaps more naturally, notice that $\xi\oplus\xi^\perp=\bd n$, so $-\xi=\xi^\perp-\bd n$. The formal difference $\xi^\perp-\bd n$ then has rank $n-k-n=-k$. Note that both definitions of $-\xi$ are equivalent: the presence of a trivial vector bundle does not affect the resulting map $X\to\BO$, only the component $X\to\Z$.
\end{example}

For all our purposes, we have seen how a virtual bundle $\eta$ of rank $r$ defines an associated Thom spectrum $X^\eta$, even when $r<0$. In this context, we have the following generalization of the classical Thom isomorphism established by Dold (details can be found in \cite[\S 4.4]{board}), as we explain subsequently.

%
%

\section{The general bundle transfer map}

\subsection{Virtual bundles and generalized Thom isomorphisms}
As usual, let $\BO$ denote the CW complex filtered by the classifying spaces $\BO(n)$. For a connected CW complex $X$, a \textbf{virtual bundle} $\xi$ over $X$ of rank $r\in\Z$ is an element of $\KO(X)\coloneqq[X_+,\BO\times\Z]$ whose projection onto $\Z$ is identically $r$. We can now define orientations of a virtual bundle with respect to any oriented spectrum $A$.

For any virtual bundle $\xi$ over $X$, we have an associated \textbf{Thom spectrum} $X^\xi$. We will write $n$ for the trivial (virtual) bundle of rank $n$ over $X$. For $n\ge 0$, $X^n=S^n\w X_+$. In particular $X^0=X_+$ is the disjoint union of $X$ with a point.

We note a potential confusion in terminology here: a genuine vector bundle $\xi$ of rank $k$ over $X$ can be considered a virtual bundle of the same rank, but the classifying map $X^\xi\to\MO$ has degree $-k$ (or codegree $k$). In this case, $\xi$ has a genuine classifying map $X\to\BO(k)$, and we use the same term for the induced map of spectra $X^\xi\to\MO$. Classes of maps to $\MO$ are more general, since a virtual bundle cannot necessarily be classified by a map $X\to\BO(n)$ for some $n$.

Let $A$ be a spectrum with unit $i\colon S^0\to A$, and let $\xi$ be a virtual bundle of rank $n$ over a connected CW complex $X$. A spectrum map $u\colon X^\xi\to A$ of degree $-n$ (or codegree $n$) is a \textbf{fundamental class} of $\xi$ (with respect to $A$) if $u$ restricts to $i$ on each fiber of $\xi$. More precisely, for each $x\in X$ the fiber of $\xi$ in $X^\xi$ can be identified with $S^n$. This gives a degree $n$ map of spectra $S^0\to X^\xi$ (whose $0$-th component is a map $S^n\to X^\xi$), and composing with $u\colon X^\xi\to A$ gives a map $S^0\to A$ of degree $0$. If $\xi$ has a fundamental class with respect to $A$, we also say \textbf{$\xi$ is $A$-oriented}. Now we can state the following generalization (by Dold) of the classical Thom isomorphism (details can be found in \cite[\S 4.4]{board}).

\begin{theorem}
	[Dold]\label{dold} Suppose a virtual bundle $\xi$ of rank $k$ over a CW complex $X$ is oriented with respect to a spectrum $A$ with unit $i\colon S^0\to A$. In addition, let $C$ be a spectrum with a left $A$-action $A\w C\to C$. Then for any virtual bundle $\eta$, there are ``Thom isomorphisms''
	\[\Phi^\xi\colon \rew C^*(X^\eta)\to \rew C^{*+k}(X^{\eta+\xi})\and \Phi_\xi\colon \rew C_*(X^{\eta+\xi})\to \rew C_{*-k}(X^\eta). \]
	In particular, when $\eta=0$, these isomorphisms become
	\[\Phi^\xi\colon C^*(X)\to \rew C^{*+k}(X^\xi)\and \Phi_\xi\colon \rew C_*(X^\xi)\to C_{*-k}(X). \]
\end{theorem}

When $X$ is a compact smooth manifold with tangent bundle $\tau$, we say that \textbf{$X$ is $A$-oriented} if $-\tau$ is $A$-oriented. This definition is (naturally) equivalent to the more classical version using local homology groups (\cite{board} $4.7$). This also allows for a notion of $A$-orientation for certain bundles whose fibers are smooth manifolds. Suppose $\pi\colon E\to B$ is a bundle whose fiber $F$ is a closed $k$-manifold and whose structure group $G$ is a compact Lie group acting smoothly on $F$. We say that $E$ has an $A$-structure if the virtual bundle $-\tau$ is $A$-oriented, where $\tau$ is the bundle of tangent vectors along the fibers of $\pi$ and $-\tau$ is the stable complement. Thus, an $A$-structure on $E$ amounts to appropriately compatible $A$-orientations of each fiber $\pi\inv(b)$ (for all $b\in B$).

\subsection{The Thom map associated to a bundle}

Suppose for the remainder of this section that we have a fiber bundle $\pi\colon E\to B$ whose fiber $F$ is a closed $k$-manifold and whose structure group $G$ is a compact Lie group acting smoothly on $F$. Again, let $\tau$ be the bundle of tangent vectors along the fibers of $\pi$ and let $-\tau$ denote the stable complement. I.e., $\tau$ is a vector bundle over $E$ whose fiber over $x\in E$ is the tangent space of $\pi\inv(\pi(x))\cong F$ at $x$.

In this setting, we have a transfer $T(\pi)\colon \Sigma^kB_+\to E^{-\tau}$. We summarize the construction here and refer to \cite{board} $6.20$ for more details. First consider the case when $B$ is a compact CW complex (then $E$ is compact: it has an induced CW structure from $F$ and $B$ with finitely many cells). For some $d$, there is a representation space $\R^d$ along with a $G$-equivariant smooth embedding $F\hookrightarrow\R^d$ (\cite{board} $6.19$). Write $E_G$ for the principal $G$-bundle underlying $\pi\colon E\to B$, and consider the associated $\R^d$-bundle $\eta=E_G\times_G\R^d$. By compactness, $\eta$ embeds (via a bundle map) into a trivial vector bundle $B\times\R^{n+k}$ for some $n$. Let $U$ be a tubular neighborhood of $F$ in $\R^d$ and consider the bundle $E_G\times_G U$. We now have the bundle maps shown.

\begin{figure}[h!]\centering
	\begin{tikzcd}
		F\ar[d]\ar[r,hook] & U\ar[d]\ar[r,hook] & \R^d\ar[d]\ar[r,hook] & \R^{n+k}\ar[d]\\
		E\ar[r,hook] & E_G\times_G U\ar[r,hook] & E_G\times_G\R^d \ar[r,hook] & B\times\R^{n+k}
	\end{tikzcd}
\end{figure}

We may consider $E_G\times_G U$ to be a fiberwise tubular neighborhood of $E$ in $B\times\R^{n+k}$. In addition, we can identify $E_G\times_G U$ with the unit disk bundle of the normal bundle $\nu$ of $E\hookrightarrow B\times\R^{n+k}$. Collapse the complement of $E_G\times_G U$ to a point and thus obtain a map $\Sigma^{n+k}B_+\to E^\nu$ on Thom spaces. Note that $\nu$ here is a genuine bundle of rank $n$ over $E$.

For our purposes, we must modify this construction to apply when $B$ is only filtered by compact CW complexes $B_0\subset B_1\subset B_2\subset\cdots$. Then each $B_i$ has an associated bundle $\nu_i$ over $E_i\coloneqq\pi\inv(B_i)$, say with rank $n_i$, and we can ensure these be compatible in the sense that the $\nu_i$ form a virtual bundle $\nu$ (of rank $0$) over $E$. Then the maps $\Sigma^{n_i+k}B_+\to E^{\nu_i}$ form a spectrum-level map $\Sigma^kB_+\to E^{-\tau}$, where $-\tau$ is considered a virtual bundle over $E$ (of rank $0$). For more details see \cite{fuhr}.


\subsection{Constructing the umkehr map from the Thom map}

\begin{definition}[\cite{board} $V.6.2$]\label{defmult}
  Let $A$ be a ring spectrum and let $C$ be a spectrum with a left $A$-action $A\w C\to C$. The \textbf{transfer maps} associated to a map $f\colon X\to Y$ of spaces refer to functorial homomorphisms
  \[f_!\colon C^*X\to C^{*-r}Y\qquad\text{and}\qquad f^!\colon C_*Y\to C_{*+r}X\]
  which are multiplicative in the sense that $f_!$ and $f^!$ are maps of $C_*$-modules and, for $\A\in C^*X$, $\B\in C^*Y$, and $y\in C_*Y$, we have
  \begin{itemize}
    \item [(a)] $f_!(\A\smile f^*\B)=f_!(\A)\smile\B$;
    \item [(b)] $f_!(f^*(\A)\smile\B)=(-1)^{r|\A|}\A\smile f_!\B$;
    \item [(c)] $f^!(y\frown \A)=f^!(y)\frown f^*(\A)$;
    \item [(d)] $f_*(f^!(y)\frown\A)=(-1)^{r|y|}y\frown f_!(\A)$;
    \item [(e)] $\langle f^!x,\A\rangle=(-1)^{r|x|}\langle x,f_!\A\rangle$.
  \end{itemize}
\end{definition}

\color{black}

\begin{theorem}[\cite{board} $V.6.21$ and $V.6.2$]\label{transfer}
	Let $\pi\colon E\to B$ be a bundle with fiber $F$ and structure group $G$ satisfying the following:
	\begin{itemize}
      \item $E$ and $B$ are CW complexes;
      \item $F$ is a compact smooth manifold of dimension $k$;
      \item $G$ is a compact Lie group which acts smoothly on $F$;
      \item there is a ring spectrum $A$ and a spectrum $C$ with a left $A$-action such that $-\tau$ is $A$-oriented, where $\tau$ is the bundle of tangents along the fibers of $\pi$.
    \end{itemize}
Then we have transfer maps
\[\pi_!\colon C^*(E)\to C^{*-k}(B)\and\pi^!\colon C_*(B)\to C_{*+k}(E)\]
which are multiplicative in the sense of definition \ref{defmult} and are the respective compositions
\[C^*(E) \xrightarrow{\Phi^{-\tau}} \rew C^*(E^{-\tau}) \xrightarrow{T(\pi)^*} \rew C^*(\Sigma^kB_+)\]
and
\[\rew C_{*+k}(
\Sigma^kB_+)\xrightarrow{T(\pi)_*} \rew C_{*+k}(E^{-\tau}) \xrightarrow{\Phi_{-\tau}} C_{*+k}(E).\]
\end{theorem}

%

\color{black}

\subsection{The $\HP^2$-bundle transfer}

Let us recall an earlier observation: $\PSp(3)$ acts transitively on $S^{11}\subset\H^3$ and this descends to a transitive action on $\HP^2$. The fiber over a point is $\PSp(2,1)\coloneqq \op{P}(\Sp(2)\times\Sp(1))$, so we have a bundle $\PSp(2,1)\to\PSp(3)\to\HP^2$. In turn this yields a bundle $\pi\colon\BPSp(2,1)\to\BPSp(3)$ with fiber $\HP^2$ and structure group $\PSp(3)$ acting by isometries on $\HP^2$. Let $\tau$ be the bundle of tangent vectors along the fibers of $\pi$ and write $-\tau$ for the stable complement. We now compute some basic (co)homological properties of this bundle.

\begin{theorem}\label{sixstep}
	For the bundle $\HP^2\to\BPSp(2,1)\xrightarrow{\pi}\BPSp(3)$ with $\tau$ the associated bundle of tangent vectors along the fiber,
	\begin{enumerate}
	  \item $H^*\BPSp(2,1)\cong\Z[u_2,u_3,u_4,u_8]$, where $\deg(u_i)=i$;
	  \item $H^*\BPSp(3)\cong\Z_2[t_2,t_3,t_8,t_{12}]$, where $\deg(t_j)=j$;
	  \item $\pi^*(t_2)=u_2$ and $\pi^*(t_3)=u_3$, while $\pi^*(t_8)=u_4^2+u_8$ and $\pi^*(t_{12})=u_4u_8$;
	  \item $\Sq^1(t_2)=t_3$ and $\Sq^1(t_3)=\Sq^1(t_4)=\Sq^1(t_{12})=0$;
	  \item $\Sq^2(t_2)=t_2^2$, $\Sq^2(t_3)=t_2t_3$, $\Sq^2(t_8)=0$, and $\Sq^2(t_{12})=t_2t_{12}$;
	  \item $w(\tau)=1+(u_2^2+u_4)+(u_2u_4+u_3^2)+u_3u_4+u_8$;
	  \item $\pi_!\colon H^*\BPSp(2,1)\to H^{*-8}\BPSp(3)$ is  given modulo $t_{12}$ by
	  		\[\pi_!(u_3^au_2^bu_4^cu_8^d)=\begin{cases}
				t_3^at_2^bt_8^{c/2-1)}\pw c>0\text{ is even and }d=0,\\
				t_3^at_2^bt_8^{d-1}\pw c=0\text{ and }d>0,\\
				0\pw\text{ otherwise.}
			\end{cases}\]
	\end{enumerate}
\end{theorem}

We delay the proof of these facts and show how they apply. Our $\HP^2$-bundle has a fiberwise $\Spin$-structure, i.e., $-\tau$ is oriented with respect to $\MSpin$. Applying Boardman's construction (\cite{board}) to the fiber sequence $\HP^2\to\BPSp(2,1)\xrightarrow{\pi}\BPSp(3)$ yields a bundle transfer map  $\m T\colon\Sigma^8\BPSp(3)_+\to M(-\tau)$. Meanwhile, theorem \ref{dold} gives identifications
\[\rew H^*M(-\tau)\cong H^*\BPSp(2,1)\and \rew H_*M(-\tau)\cong H_*\BPSp(2,1).\]
In terms of spectra, the transfer of theorem \ref{transfer} with $C=\MSpin$ is then
\[\MSpin\w\Sigma^8\BPSp(3)_+\xrightarrow{1\w \m T}\MSpin\w M(-\tau)\to \MSpin\w\BPSp(2,1)_+, \]
where the righthand map induces the Thom isomorphism. The Thom isomorphism can be described more explicitly as the map induced on homotopy groups of the composition indicated below. We use $\op{G}=\PSp(3)$ and $\op{H}=\PSp(2,1)$; dashed arrows indicate a map induced on homotopy groups.

\begin{figure}[h!]\centering
\begin{tiny}
	\begin{tikzcd}[row sep=50]
		\MSpin\w\Sigma^8BG_+\rar["1\w \m T"]\ar[drrr,dashed,"\pi^!" {above=15, left=30}]\ar[rr,bend left=17,"1\w t"]\ar[rrr,bend left=23,"T"]\ar[rrr,bend left=30,"\Psi",dashed]
			&\MSpin\w M(-\tau)\ar[drr,"\Phi",dashed,bend left=7]\rar["1\w \m M"]\ar[d,"1\w\D" below left=2,crossing over]
			& \MSpin\w \MSpin\rar["\mu"] & \MSpin\\
		& |[xshift=-10pt]|\mathllap{\MSpin\w}M(-\tau)\w BH_+ \rar["1\w \m M\w 1" below]
			& \MSpin\w \MSpin\w BH_+\rar["\mu\w 1" below]
			& \mathrlap{\MSpin\w BH_+}\hspace{1.25cm}\uar["\text{proj}"]
	\end{tikzcd}
\end{tiny}
\end{figure}
Here $\m M$ is the map of Thom spectra induced by the classifying map $\BH\to\BSpin$ of $-\tau$, which motivates theorem \ref{sixstep} (in particular, it demonstrates why computing $w(-\tau)$ is relevant).

As mentioned earlier, we will use the same names for several of these maps with $\MSpin$ replaced with $\MSpinc$. For example, $T\colon \Sigma^8\BPSp(3)\to\MSpinc$ and $\TC\colon\Sigma^4\BSU(3)\to\MSpinc$ together are used to prove theorem \ref{mymain}.

\section{The $\CP^2$-bundle transfer}

Analogously to the quaternionic case, the unitary group $\U(3)$ acts transitively on $S^5\subset\C^3$ yielding a transitive action on $\CP^2$. As before, the stabilizer of $[0:0:1]$ is $\SU(2,1)\coloneq\op{S}(\U(2)\times\U(1))$ and we hence obtain a bundle $\BSU(2,1)\to\BSU(3)$ with fiber $\CP^2$. This bundle admits a $\Spinc$ structure, giving a map $\BSU(3)\to\BSpinc$. Now a class in $\Omega_n^\Spinc\BSU(3)$ consists of an $n$-manifold $P$ with a stable normal $\Spinc$ structure along with a map $f\colon P\to \BSU(3)$. Let $\widehat P$ be the pullback of the bundle $\BSU(2,1)\to\BSU(3)$; hence $\widehat P$ is an $(n+4)$-dimensional $\Spinc$ manifold fibered over $P$. On bordism classes, this correspondence $[P]\to [\widehat P]$ geometrically describes the transfer map $\Omega_n^\Spinc\BSU(3)\to \Omega_{n+4}^\Spinc$.


We have the following maps of spectra.
\begin{figure}[h!]\centering
	\begin{tikzcd}
		\MSpinc\w\Sigma^4BG_+\ar[r,"1\w t"]\ar[rr,"T" below,bend right=15] & \MSpinc\w\MSpinc\ar[r,"\mu"] & \MSpinc\ar[r,"D"] & \ku
	\end{tikzcd}
\end{figure}

These induce maps on homology.
\begin{equation}\label{1}
	H_*\MSpinc\T H_*\Sigma^4BG_+\xrightarrow{1\T t_*} H_*\MSpinc\T H_*\MSpinc\xrightarrow{\mu} H_*\MSpinc\xrightarrow{D_*} \ku
\end{equation}


%
%

Recall $\m A_*=\Z_2[\xi_1,\xi_2,\xi_3,\ldots]$, where $\deg(\xi_i)=2^i-1$ with coproduct
\[\psi(\xi_n)=\sum_{i=0}^n\xi_{n-i}^{2^i}\T\xi_i=\xi_n\T 1+\xi_{n-1}^2\T\xi_1+\cdots+\xi_1^{2^{n-1}}\T\xi_{n-1}+1\T\xi_n.\]
Alternatively, one can use the Hopf algebra conjugates $\zeta_i$ of $\xi_i$, characterized by
\[\psi(\zeta_n)=\sum_{i=0}^n\zeta_i\T\zeta_{n-i}^{2^i}=1\T\zeta_n+\zeta_1\T\zeta_{n-1}^2+\cdots+\zeta_{n-1}\T\zeta_1^{2^{n-1}}+\zeta_n\T 1. \]
The dual of $\Sq^n=\Sq(n)$ is $\xi_1^n$, while $Q_n=\Sq(0,\ldots, 0,1)$ (with $n$ zeros followed by a $1$) is dual to $\xi_{n+1}$.

Note that $H_*H\Z_2$ is also the dual $\m A_*$ of the Steenrod algebra $\m A$. For the subalgebra $E(1)$ generated by $Q_0$ and $Q_1$, we have $H^*\ku=\m A\hmod E(1)=\Z_2\T_{E(1)}\m A$. Dually, $H_*\ku=\Z_2\square_{E(1)_*}\m A_*$.

Theorem \ref{reducedc} stated that, if $Y$ is an $\MSpinc$-module spectrum for which $H_*Y$ is bounded below and of finite type, there is a functorial isomorphism of $\m A_*$-comodules
\[H_*Y\to \m A_*\square_{E(1)_*}\ol{H_*Y}.\]
Applying theorem \ref{reducedc} to (\ref{1}) gives
\[\ol{H_*\MSpinc}\T H_*\Sigma^4 BG_+\xrightarrow{1\T\ol t_*}\ol{H_*\MSpinc}\T H_*\MSpinc\xrightarrow{\mu}\ol{H_*\MSpinc}\xrightarrow{\ol D_*}\ol{H_*\ku}. \]
However, $\ol{H_*\ku}=\Z_2$, and one can identify $\ol D_*$ with the augmentation homomorphism of $\ol{H_*\MSpinc}$. It follows that $\ker\ol D_*$ is generated over $\Z_2$ by elements which have at least one nontrivial indecomposable factor. These elements can be written $\mu(x\T y)$ for some $x\in \ol{H*\MSpinc}$ and $y\in Q\ol{H_*\MSpinc}$. (Here, $x$ may be $1$, and, for an algebra $A$ with augmentation ideal $I(A)$, $QA$ denotes the indecomposable quotient $I(A)/\mu(I(A)\T I(A))$.) Thus to show that $\ol T_*$ surjects onto $\ker D_*$, it suffices to show that the composition $H_*\Sigma^4 BG_+\to H_*\MSpinc\to \ol{H_*\MSpinc}\to Q\ol{H_*\MSpinc}$ is surjective.

We have $QH_n\MSpinc\cong\Z_2$ for $n\ge 2$, $n\neq 2^k+1$ and $QH_n\MSpinc=0$ otherwise. Further, the projection $QH_n\MSpinc\to Q\ol{H_n\MSpinc}$ is an isomorphism for $n\ge 4$, $n\neq 2^k\pm 1$, and $Q\ol{H_n\MSpinc}=0$ for $n<4$ or $n=2^k\pm 1$. We have
	\[\ol{H_*\MSpinc}=\Z_2\left[x_n^{\B(n)}:n\ge 4, n\neq 2^k\pm 1\right], \]
	where $\B(n)=2$ if $\A(n)<3$ and $\B(n)=1$ for $\A(n)\ge 3$ (here $\A(n)$ is the number of nonzero terms in the base-$2$ expansion of $n$). Ordering the generators by their lower indices,
	\[H_*\MSpinc=\Z_2\left[x_1^2,x_2^2,x_3^2,x_4^2,x_5^2,x_6^2,x_7,x_8^2,x_9^2,x_{10}^2,x_{11},\ldots \right]. \]
	Since $R=\Z_2[x_1^2,x_3^2,x_7,x_{15},\ldots]$, we have
	\[\ol{H_*\MSpinc}=\Z_2\T_R H_*\MSpinc=\Z_2\left[x_2^2,x_4^2,x_5^2,x_6^2,x_8^2,x_9^2,x_{10}^2,x_{11},\ldots \right]. \]

We must show that $H_n\Sigma^4BG_+\to QH_n\MSpinc$ is surjective for $n\ge 4$, $n\neq 2^k\pm 1$. Dually, we must show $PH^n\MSpinc\to H^{n-4}BG_+$ is injective for $n\ge 4$, $n\neq 2^k\pm 1$. Recall that the map $BG_+\to \MSpinc$ can be decomposed into $\Sigma^4BG_+\xrightarrow{T(\pi)}M(-\tau)\xrightarrow{M(c)}\MSpinc$, where $\pi\colon BH\to BG$ is the bundle map, $T(\pi)$ is the Thom collapse map, and $c\colon BH\to \BSpinc$ is classifies the complement $-\tau$ of $\tau$, the bundle along the fibers of $\pi$. The $h$-space inverse $\BSpinc\to\BSpinc$ provides a homotopy equivalence $\MSpinc\to\MSpinc$ allowing us to replace $M(\-\tau)$ with $M(\tau)$. On cohomology, we have the maps
\[H^n\MSpinc\xrightarrow{c(\tau)^*} H^nBH\xrightarrow{T(\pi)^*} H^{n-4}BG. \]

\begin{itemize}
  \item $H^*BH=\Z_2[x_2,x_4]$,
  \item $H^*BG=\Z_2[y_4,y_6]$,
  \item $w(\tau)=1+x_2+x_4$,
  \item $\pi^*y_4=x_2^2+x_4$ and $\pi^*y_6=x_2x_4$,
  \item we can write $H^*BH$ as a free $H^*BG$-module with basis $1,x_2,x_2^2$,
  \item and the transfer map $H^nBH\to H^{n-4}BG$ is the $H^*BG$-module map taking $x=r_0(x)+r_1(x)x_2+r_2(x)x_2^2$ to $r_2(x)$, where $r_i(x)\in H^*BG$.
\end{itemize}

 


\begin{appendicesmulti}

\chapter{The $\m A(1)$-action on $H^*\BPSp(3)$}

In this section, let $G=\PSp(3)$ be the quotient of $\Sp(3)$ by its center $\pm I$ (here $I$ is the identity matrix), and let $H=\PSp(2,1)$ be the quotient of $\Sp(2)\times \Sp(1)$ by $\pm I$. Notice that $H$ is a subgroup of $G$ and the inclusions
\[\Sp(1,1,1)\to \Sp(2,1)\to \Sp(3)\]
induce inclusions
\[i_1\colon \PSp(1,1,1)\to H\and i_2\colon H\to G\]
(here $\PSp(1,1,1)\coloneqq\op{P}(\Sp(1)^3)$). Let $i\colon \Z_2^4\to \PSp(3)$ be the composition
\[\Z_2^2\times\Z_2^2\xrightarrow{j\times 1} \PSp(1)\times\Z_2^2\xrightarrow{\Delta\times j_1\times j_2} \PSp(1,1,1), \]
where $j$ maps $(1,0)$ to $i$ and $(0,1)$ to $j$, $\Delta$ is the diagonal, and $j_n$ sends each generator to $-1$ in the $n$-th factor. Thus,
\begin{align*}
	(1,0,0,0)&\mapsto [i,i,i]\\
	(0,1,0,0)&\mapsto [j,j,j]\\
	(0,0,1,0)&\mapsto [-1,1,1]\\
	(0,0,0,1)&\mapsto [1,-1,1],
\end{align*}
etc. For example, $(1,1,0,1)\mapsto [k,-k,k]$. (Note that the images of $[i,i,i]$ and $[j,j,j]$ indeed commute in $\PSp(1)^3)$.) Writing $\Z_2$ as the multiplicative group $\pm 1$, we have
\[(x_1,x_2,y_1,y_2)\mapsto [(-1)^{y_1}i^{x_1}j^{x_2},(-1)^{y_2}i^{x_1}j^{x_2},i^{x_1}j^{x_2}]. \]

Given a compact Lie group $K$, any representation $\rho\colon K\to \GL_n(\R)$ gives rise to a real vector bundle over $BK$ which we denote $E\rho$. Define a four-dimensional real representation $R_{ij}$ of $\PSp(1,1,1)$ acting on $\H$ via $[h_1,h_2,h_3]\cdot x=h_ix\ol{h_j}$. If $i=j$, this action fixes the $\R$-span of $1$, so we can decompose $R_{ii}$ as the sum of a trivial representation and a $3$-dimensional representation $R_i$. The inclusion $\Z_2^4\times\Z_2^2\xrightarrow{i}\PSp(3)$ then defines a bundle over $B((\Z/2)^2\times(\Z/2)^2)\cong (\R P^\infty)^4$, and we can identify $w(ER_{ij})$ with its image in $H^*((\R P^\infty)^4;\Z/2)\cong \Z/2[x_1,x_2,y_1,y_2]$.
\\

We will now compute how $\rho_{ij}$ transforms the elements $x_i,y_j$. In general, we have
\[ \rho_{ij}(x_1,x_2,y_1,y_2)(h)=[(-1)^{y_1}i^{x_1}j^{x_2},(-1)^{y_2}i^{x_1}j^{x_2},i^{x_1}j^{x_2}]\cdot h. \]

Case 1: $i=j$. Then
\begin{align*}
	\rho_{ii}(x_1,x_2,y_1,y_2)(h)&=[(-1)^{y_1}i^{x_1}j^{x_2},(-1)^{y_2}i^{x_1}j^{x_2},i^{x_1}j^{x_2}]\cdot h\\
	&=i^{x_1}j^{x_2}h\ol{i^{x_1}j^{x_2}}\\
	&=i^{x_1}j^{x_2}h\ol{j^{x_2}}\ol{i^{x_1}}\\
	&=(-1)^{x_1+x_2}i^{x_1}j^{x_2}hj^{x_2}i^{x_1}.
\end{align*}
In particular,
\[\rho_{ii}(x_1,x_2,y_1,y_2)(1)=(-1)^{x_1+x_2}(i^{x_1}(j^{x_2}j^{x_2})i^{x_1})=1.\]
When $h=i$, note that $jij=jk=i$, so
\begin{align*}
	\rho_{ii}(x_1,x_2,y_1,y_2)(i)&=(-1)^{x_1+x_2}i^{x_1}(j^{x_2}ij^{x_2})i^{x_1}\\
	&=(-1)^{x_1+x_2}(i^{x_1})i(i^{x_1})\\
	&=(-1)^{x_1+x_2}i^{2x_1+1}\\
	&=(-1)^{x_2}i.
\end{align*}
Next
\begin{align*}
	\rho_{ii}(x_1,x_2,y_1,y_2)(j)&=(-1)^{x_1+x_2}i^{x_1}(j^{x_2}jj^{x_2})i^{x_1}\\
	&=(-1)^{x_1+x_2}(i^{x_1})j^{2x_2+1}(i^{x_1})\\
	&=(-1)^{x_1}i^{x_1}ji^{x_1}\\
	&=(-1)^{x_1}j
\end{align*}
and
\begin{align*}
	\rho_{ii}(x_1,x_2,y_1,y_2)(k)&=(-1)^{x_1+x_2}i^{x_1}(j^{x_2}kj^{x_2})i^{x_1}\\
	&=(-1)^{x_1+x_2}(i^{x_1})k(i^{x_1})\\
	&=(-1)^{x_1+x_2}k.
\end{align*}

Thus the total Stiefel-Whitney class of $i^*(ER_{ii})$ is thus
\[Bi^*(w(ER_{ii}))=(1+x_1)(1+x_2)(1+x_1+x_2). \]

Case 2: $i\neq j$. We have
\begin{align*}
	\rho_{23}(x_1,x_2,y_1,y_2)(h)&=(-1)^{y_2}\rho_{11}(x_1,x_2,y_1,y_2)(h)\\
	\rho_{13}(x_1,x_2,y_1,y_2)(h)&=(-1)^{y_1}\rho_{11}(x_1,x_2,y_1,y_2)(h)\\
	\rho_{12}(x_1,x_2,y_1,y_2)(h)&=(-1)^{y_1+y_2}\rho_{11}(x_1,x_2,y_1,y_2)(h).
\end{align*}

Thus, if we let $y_3=y_1+y_2$,
\begin{align*}
	Bi^*(w(ER_{23}))&=(1+y_2)(1+x_1+y_2)(1+x_2+y_2)(1+x_1+x_2+y_2)\\
	Bi^*(w(ER_{13}))&=(1+y_1)(1+x_1+y_1)(1+x_2+y_1)(1+x_1+x_2+y_1)\\
	Bi^*(w(ER_{12}))&=(1+y_3)(1+x_1+y_3)(1+x_2+y_3)(1+x_1+x_2+y_3).
\end{align*}

Next we compute $w(\BG)$ and the action by $\Sq^1, \Sq^2$. First,
\[Bi^*(w(ER_{ii}))=1+(x_1^2+x_1x_2+x_2^2)+x_1x_2(x_1+x_2)\pmod 2. \]
Let $t_2=x_1^2+x_1x_2+x_2^2$ and $t_3=x_1x_2(x_1+x_2)$ so that $w(i^*ER_{ii})=1+t_2+t_3$. When $i\neq j$, we define as $y_\ell$,
\begin{align*}
	Bi^*(w(ER_{ij}))&=1+x_1^2+x_1x_2+x_2^2+x_1x_2(x_1+x_2)+x_1x_2(x_1+x_2)y_\ell+(x_1^2+x_1x_2+x_2^2)y_\ell^2+y_\ell^4\\
	&=1+t_2+t_3+t_3y_\ell+t_2y_\ell^2+y_\ell^4.
\end{align*}
We can write this fourth order term as $s_k=t_3y_\ell+t_2y_\ell^2+y_\ell^4$, where $\{i,j,k\}=\{1,2,3\}$. To summarize,
\begin{equation}\label{1}
	Bi^*(w(ER_{ij}))=\begin{cases}
	1+t_2+t_3\pw i=j\\
	1+t_2+t_3+s_k\pw \{i,j,k\}=\{1,2,3\}
\end{cases}
\end{equation}
where
\begin{align*}
	s_1&=t_3y_2+t_2y_2^2+y_2^4\\
	s_2&=t_3y_1+t_2y_1^2+y_1^4\\
	s_3&=t_3y_3+t_2y_3^2+y_3^4.
\end{align*}

\begin{remark}
	There seems to be a minor indexing error in Stolz's work. He claims $\rho_{23}=(-1)^{y_1}\rho_{11}$ and $\rho_{13}=(-1)^{y_2}\rho_{11}$. Then $s_k=t_3y_k+t_2y_k^2+y_k^4$, which cleans up some notation. This essentially amounts to switching the role of $y_1$ and $y_2$ in the map $i$. In his version, $(0,0,1,0)$ would map to $[1,-1,1]$ and $(0,0,0,1)$ maps to $[-1,1,1]$.
\end{remark}

Next we consider the adjoint representation $\f g$ of $G$, which is equivalent to the conjugation action of $G$ skew-Hermitian $3\times 3$ quaternionic matrices.

We claim that, restricted to $P(\Sp(1)^3)$, the representation $\f g$ decomposes as
\[R_1\oplus R_2\oplus R_3\oplus R_{23}\oplus R_{13}\oplus R_{12}.\]

The restriction here is from the $i_2\circ i_1$, where

\[(\Z/2)^2\times(\Z/2)^2\xrightarrow{i} P(\Sp(1)^3)\xrightarrow{i_1} P(\Sp(2)\times\Sp(1))\xrightarrow{i_2} P\Sp(3). \]

It follows that

\[B(i_2i_1i)^*(w(E\f g))=t^3(t+s_1)(t+s_2)(t+s_3), \]

where $t=1+t_3+t_3$. To simplify this, it helps to notice that
\begin{align*}
	s_1+s_2&=t_3(y_1+y_2)+t_2(y_1^2+y_2^2)+y_1^4+y_2^4\\
	&=t_3(y_1+y_2)+t_2(y_1+y_2)^2+(y_1+y_2)^4\\
	&=s_3.
\end{align*}

We have
\[B(i_2i_1i)^*(w(E\f g))=(s_1^2s_2+s_1s_2^2)t^3+(s_1^2+s_1s_2+s_2^2)t^4+t^6 \]

and we let $t_8=s_1^2+s_1s_2+s_2^2$ and $t_{12}=s_1^2s_2+s_1s_2^2$ so that
\begin{equation}\label{2}
	B(i_2i_1i)^*(w(E\f g))=t_{12}t^3+t_8t^4+t^6.
\end{equation}

Now we can put this together. From (\ref{1}), we know that $t_2,t_3$ are in the image of $Bi^*$. We can consider $Bi_1^*$ and $Bi_2^*$ as bundles with respective fibers $\H P^1$ and $\H P^2$, and Hurewicz's theorem shows that $Bi_1^*$ and $Bi_2^*$ are isomorphisms on cohomology groups of degree at most $3$. In particular, $t_2$ and $t_3$ are in the image of $Bi_2i_1i^*$. Using Kono's computation of $H^*\BPSp(3)$, we can thus identify the generators $t_2$ and $t_3$ with those in $H^*\BPSp(3)$.
\\

On the other hand, (\ref{2}) shows $t_8$ and $t_{12}$ are also in the image of $Bi_2i_1i^*$. These elements are polynomials in $s_k$, so they are not in the polynomial ring $\Z/2[t_2,t_3]$. We can thus identify $t_8$ and $t_{12}$ with the generators in $H^*\BPSp(3)$.

\section{Action by $\m A(1)$ on $\BPSp(3)$}

In the last section, we identified $H^*\BPSp(3)=\Z/(2)[t_2,t_3,t_8,t_{12}]$ with the subring of $\Z_2[x_1,x_2,y_1,y_2]$ via
\begin{align*}
	t_2&=x_1^2+x_1x_2+x_2^2\\
	t_3&=x_1x_2(x_1+x_2)\\
	t_8&=s_1^2+s_1s_2+s_2^2\\
	t_{12}&=s_1s_2(s_1+s_2)
\end{align*}

where

\[s_1=t_3y_2+t_2y_2^2+y_2^4\and s_2=t_3y_1+t_2y_1^2+y_1^4. \]

For a generator $x\in H^*B\Z_2$, we have $\Sq^1 x=x^2$ and $\Sq^2 x=0$. We use this to compute the actions of $\Sq^1$ and $\Sq^2$ on cohomology classes.

\begin{theorem}\label{sqc}
	Identifying $H^*\BPSp(3)=\Z_2[t_2,t_3,t_8,t_{12}]$, we have
	\begin{align*}
	&\Sq^1t_2=t_3& &\Sq^1t_3=0& &\Sq^1t_8=0& &\Sq^1t_{12}=0&\\
	&\Sq^2t_2=t_2^2& &\Sq^2t_3=t_2t_3& &\Sq^2t_8=0& &\Sq^2t_{12}=t_2t_{12}.&
\end{align*}
\end{theorem}
\begin{proof}
Using naturality, we compute
\begin{align*}
	\Sq^1(t_2)&=\Sq^1(x_1^2+x_1x_2+x_2^2)\\
	&=\Sq(x_1)x_2+x_1\Sq(x_2)\\
	&=x_1^2x_2+x_1x_2^2\\
	&=t_3
\end{align*}
and
\begin{align*}
	\Sq^2(t_2)&=\Sq^2(x_1^2+x_1x_2+x_2^2)\\
	&=x_1^4+\Sq^2(x_1)x_2+\Sq^1(x_1)\Sq^1(x_2)+x_1\Sq^2(x_2)+x_2^4\\
	&=x_1^4+x_1^2x_2^2+x_2^4\\
	&=(x_1^2+x_1x_1+x_2^2)^2\\
	&=t_2^2
\end{align*}
(which is expected since $t_2$ has degree $2$). Next
\begin{align*}
	\Sq^1(t_3)&=\Sq^1(x_1x_2(x_1+x_2))\\
	&=\Sq^1(x_1x_2)(x_1+x_2)+x_1x_2(\Sq^1(x_1)+\Sq^1(x_2))\\
	&=(x_1^2x_2+x_1x_2^2)(x_1+x_2)+x_1x_2(x_1^2+x_2^2)\\
	&=x_1x_2(x_1+x_2)^2+x_1x_2(x_1^2+x_2^2)\\
	&=x_1x_2(x_1^2+x_2^2)+x_1x_2(x_1^2+x_2^2)\\
	&=0
\end{align*}
and
\begin{align*}
	\Sq^2(t_3)&=\Sq^2(x_1x_2(x_1+x_2))\\
	&=\Sq^2(x_1x_2)(x_1+x_2)+\Sq^1(x_1x_2)\Sq^1(x_1+x_2)+x_1x_2\Sq^2(x_1+x_2)\\
	&=(x_1x_2)^2(x_1+x_2)+(x_1^2x_2+x_1x_2^2)(x_1^2+x_2^2)\\
	&=(x_1x_2)^2(x_1+x_2)+(x_1x_2)(x_1+x_2)^3\\
	&=t_3(x_1x_2+(x_1+x_2)^2)\\
	&=t_2t_3.
\end{align*}

Continuing,
\begin{align*}
	\Sq^1(s_1)&=\Sq^1(t_3y_2+t_2y_2^2+y_2^4)\\
	&=\Sq^1(t_3)y_2+t_3y_2^2+\Sq^1(t_2)y_2^2\\
	&=t_3y_2^2+t_3y_2^2\\
	&=0
\end{align*}
and similarly $\Sq^1(s_2)=0$. Further,

\begin{align*}
	\Sq^2(s_1)&=\Sq^2(t_3y_2+t_2y_2^2+y_2^4)\\
	&=\Sq^2(t_3)y_2+t_3\Sq^2(y_2)+t_2^2y_2^2+t_2\Sq^2(y_2^2)+\Sq^2(y_2^4)\\
	&=t_2t_3y_2+t_2^2y_2^2+t_2y_2^4+\Sq^2(y_2^2)y_2^2+\Sq^1(y_2^2)\Sq^1(y_2^2)+y_2^2\Sq^2(y_2^2)\\
	&=t_2t_3y_2+t_2^2y_2^2+t_2y_2^4+y_2^6+y_2^6\\
	&=t_2(t_3y_2+t_2y_2^2+y_2^4)\\
	&=t_2s_1
\end{align*}

and similarly $\Sq^2(s_2)=t_2s_2$. Finally, we can compute

\begin{align*}
	\Sq^1(t_8)&=\Sq^1(s_1^2+s_1s_2+s_2^2)\\
	&=\Sq^1(s_1)s_2+s_1\Sq^1(s_2)\\
	&=0
\end{align*}

as well as

\begin{align*}
	\Sq^2(t_8)&=\Sq^2(s_1^2+s_1s_2+s_2^2)\\
	&=\Sq^2(s_1)s_1+\Sq^1(s_1)\Sq^1(s_1)+s_1\Sq^2(s_1)\\
	&\qquad +\Sq^2(s_1)s_2+\Sq^1(s_1)\Sq^1(s_2)+s_1\Sq^2(s_2)\\
	&\qquad +\Sq^2(s_2)s_2+\Sq^1(s_2)\Sq^1(s_2)+s_2\Sq^2(s_2)\\
	&=\Sq^2(s_1)s_2+s_1\Sq^2(s_2)\\
	&=t_2s_1s_2+t_2s_1s_2\\
	&=0.
\end{align*}

Then

\begin{align*}
	\Sq^1(t_{12})&=\Sq^1(s_1s_2(s_1+s_2))\\
	&=\Sq^1(s_1s_2)(s_1+s_2)+s_1s_2\Sq^1(s_1+s_2)\\
	&=0
\end{align*}

and

\begin{align*}
	\Sq^2(t_{12})&=\Sq^2(s_1s_2(s_1+s_2))\\
	&=\Sq^2(s_1s_2)(s_1+s_2)+\Sq^1(s_1s_2)\Sq^1(s_1+s_2)+s_1s_2\Sq^2(s_1+s_2)\\
	&=(\Sq^2(s_1)s_2+\Sq^1(s_1)\Sq^1(s_2)+s_1\Sq^2(s_2))(s_1+s_2)+t_2s_1s_2(s_1+s_2)\\
	&=t_2t_{12}.
\end{align*}
\end{proof}

\section{Cohomology of $\BPSp(2,1)$}

We write $H=\PSp(2,1)\coloneqq \op{P}(\Sp(2)\times\Sp(1))$ and $G=\PSp(3)$. We have the bundle
\[\H P^2=G/H\to BH\xrightarrow{\pi} BG. \]

Let $\tau$ denote the corresponding vertical bundle along the fibers of $\pi$.

Also recall the maps
\[(\Z/2)^2\times(\Z/2)^2\xrightarrow{i}P(\Sp(1)^3)\xrightarrow{i_1}H\xrightarrow{i_2}G. \]
Note that $Bi_2=\pi$.

Claim 2: restricted to $H=P(\Sp(2)\times\Sp(1))$, the representation $\f g$ splits as $\f h\oplus\f h^\perp$, with $\f h$ the adjoint representation of $H$ and $\f h^\perp\cong \tau$. When restricted to $P(\Sp(1)^3)$, $\f h^\perp$ splits as $R_{13}\oplus R_{23}$.

Assuming claim $2$ holds,
\begin{align*}
	B(i_1i)^*w(\tau)&=(t+s_1)(t+s_2)\\
	&=t^2+(s_1+s_2)t+s_1s_2\\
	&=1+s_1+s_2+t_2^2+t_3^2+(s_1+s_2)t_2+(s_1+s_2)t_3+s_1s_2.
\end{align*}

Define
\[u_2=t_2\qquad u_3=t_3\qquad u_4=s_1+s_2\qquad u_8=s_1s_2. \]

Then
\[B(i_1i)^*w(\tau)=1+(u_4+u_2^2)+(u_3^2+u_2u_4)+u_3u_4+u_8. \]

This means that the class $w_4(\tau)$ restricts to a nontrivial element in $H^4\H P^2$, and thus the map $H^*BH\to H^*G/H$ is surjective. The Serre spectral sequence of $\pi$ collapses, and $\pi^*\colon H^*BG\to H^*BH$ is thus injective. Moreover, by Leray-Hirsch, $H^*BH$ is a free $H^*BG$ module with basis $\{1,w_4(\tau),w_4(\tau)^2\}$. We can therefore identify $H^*BH$ with the subring of $\Z/(2)[u_2,u_3,u_4,u_8]$ where $\pi^*(t_2)=u_2$, $\pi^*(t_3)=u_3$, $\pi^*(t_8)=u_4^2+u_8$, and $\pi^*(t_{12})=u_4u_8$. Under this identification, we have
\[w(\tau)=1+(u_4+u_2^2)+(u_3^2+u_2u_4)+u_3u_4+u_8. \]

\chapter{Primitive Generators}

Since $\RP^\infty$ is an $h$-space, the cohomology forms a Hopf algebra. As an algebra, $H^*\RP^\infty$ is the polynomial ring $\Z_2[u]$, where $\deg(u)=1$. The coalgebra structure is determined by the fact that the coproduct $\D$ is an algebra homomorphism and $\D(u)=1\T u+u\T 1$. For the algebra structure in homology, $H_*\RP^\infty=\W\left(x_1,x_2,x_4,x_8,\ldots\right)$ is the exterior algebra on generators $x_n$, where $x_n$ is the linear dual of $u^{2^n}$.

We know that $H^*BO=\Bbb Z_2[w_1,w_2,w_3,\ldots]$. Dually, $H_*\BO=\Z_2[u_1,u_2,u_3,\ldots]$, where $u_i$ is the linear dual of $w_1^i$. The canonical inclusion $j\colon \BO(1)\to \BO$ then satisfies $j_*x_n=u_{2^n}$. Clearly there is one indecomposable element in each degree, hence one primitive in each degree. In this case, the primitive elements are just $s_n$ (the symmetric polynomial generated by $x_1^n+\cdots+x_k^n$ written in terms of the elementary symmetric polynomials $w_i$, where $k\ge n$).
s

As an algebra, $H^*\BSO=\Z_2[w_n:n\ge 2]$. There is one indecomposable element in degree $n$ for every $n\ge 2$, so, dually, $P^nH^*\BSO$ is $\Z_2$ for $n\ge 2$ and $0$ otherwise. The primitives of $H^*\BSO$ are not just those for $H^*\BO$ in degrees $n\ge 2$. The reason is because $s_n$ is equal to $(-1)^{n+1}nw_n$ plus decomposables, so in even degrees we need to know that $s_n$ has a nontrivial decomposable term which is not a multiple of $w_1$. For example,
\[s_4=w_1^4-4w_1^2w_2+4w_1w_3+2w_2^2-4w_4=w_1^4\in H^*\BO,\]
so $s_4$ vanishes when restricted to $H^*\BSO$. However, we can use the following lemma.
\begin{lemma}
	If $x$ is primitive, then $x^{2^n}$ is primitive for all $n$.
\end{lemma}
\begin{proof}
	Using the Frobenius map,
	\[\D(x^{2^n})=\D(x)^{2^n}=(1\T x+x\T 1)^{2^n}=1\T x^{2^n}+x^{2^n}\T 1.\]
\end{proof}

For $n$ even, we write $n=2^km$ for some $k\ge 1$ and odd $m$. Then as long as $m\ge 1$, we know $s_m$ is a nontrivial primitive in $H^*\BSO$. The primitive in degree $n$ is therefore $s_m^{2^k}$. We still have not accounted for the primitives whose degrees are powers of $2$. For degree reasons, $w_2$ is primitive (note that $w_2\neq s_2$ since $s_2=s_1^2\pmod 2$). This means $w_2^{2^n}$ is the primitive element of degree $2^{n+1}$. We summarize this below.
\begin{theorem}\label{bsoprimitives}
	The coalgebra $H^*\BSO$ has one primitive element $v_n$ of degree $n$ for all $n\ge 2$, where
	\[v_n=\begin{cases}
		w_2\pw n=2\\
		s_n\pw n\text{ is odd}\\
		s_{n/2}^2\pw n>2\text{ is even}.\\
	\end{cases} \]
\end{theorem}
\begin{proof}
	We already computed the following:
	\[v_n=\begin{cases}
		s_n\pw n\text{ is odd}\\
		s_m^{2^k}\pw n=2^km\text{, where }m>1\text{ is odd and }k\ge 1\\
		w_2^{2^{k-1}}\pw n=2^k.
	\end{cases} \]
	If $n$ is even and not a power of $2$, we can write $n=2^km$ for $k\ge 1$ and $m>1$ odd. We then have $v_n=s_m^{2^k}$. If $k>1$, then $n/2$ is also even and not a power of $2$, so $v_{n/2}=s_m^{2^{k-1}}$. In this case $v_n=v_{n/2}^2$. Now if $k=1$ (so $n=2m$), then $v_n=s_m^2$ and $v_{n/2}=v_m=s_m$. Again we have $v_n=v_{n/2}^2$. Next suppose $n=2^k$. For all $k\ge 1$, we saw $v_n=w_2^{2^{k-1}}$. When $k>1$, $n/2=2^{k-1}$ is also an even power of $2$, so $v_{n/2}=w_2^{2^{k-2}}$. Thus $v_n=v_{n/2}^2$ as claimed.
\end{proof}

\begin{theorem}\label{bspinprimitives}
	The primitive elements of $H^*\BSpin$ comprise one generator $z_n$ for each $n\ge 4$ not of the form $2^s+1$, where
	\[z_n=\begin{cases}
		s_n\pw n\neq 2^s+1\text{ is odd}\\
		w_4^{n/4}\pw \A(n)=1\text{ and $n$ is a power of $2$}\\
		s_{m,m}^{n/(2m)}\pw \A(n)=2\text{ and $n=2^km$ for $m=2^s+1$}\\
		z_{n/2}^2\pw \A(n)\ge 3\text{ and $n$ is even}.
	\end{cases} \]
	Here $\A(n)$ is the number of $2$-bits in the binary expansion of $n$. The first definition only applies for $n\neq 2^s+1$, but this covers all odd-degree primitive generators; the third through fourth definitions together cover the all even degrees starting with $4$.
\end{theorem}
\begin{proof}
	In $H^*\BO$, $s_n$ is $nw_n$ plus decomposable elements, so $s_n$ restricts to a nontrivial primitive in $H^*\BSpin$ provided that $n$ is odd, $n\ge 4$, and $n\neq 2^s+1$ (these are the degrees in which $w_n$ is nontrivial). When $n$ is even, start by writing $n=2^km$ for $k\ge 1$ and for odd $m>1$. If $m\neq 2^s+1$ (note that this precludes $m=3$), then $s_m$ is a nontrivial primitive in $H^*\BSpin$ and we can set $z_n=s_m^{2^k}$. Equivalently, $z_n=z_{n/2}^2$. The condition $m\neq 2^s+1$ means $\A(n)\neq 2$, and $m>1$ means $\A(n)\neq 1$. Finally we consider when $m=2^s+1$. Then $n=2^k(2^s+1)=2^{k+s}+2^k$, so $\A(n)=2$. The class $s_m$ is not primitive, and in fact is zero by lemma \ref{z}. It follows that $s_{m,m}$ is primitive since
	\[\D(s_{m,m})=s_{m,m}\T 1+s_m\T s_m+1\T s_{m,m}=s_{m,m}\T 1+1\T s_{m,m}. \]
	 We can now define $z_n=s_{m,m}^{2^{k-1}}$ for these values of $n$.
\end{proof}

%


\begin{lemma}\label{z}
	For all $k$, $s_{2^k+1}$ vanishes when restricted to $H^*\BSpin$.
\end{lemma}
\begin{proof}
	Stong shows that the natural map $H^*\BSO\to\H^*\BSpin$ is epic with kernel $\m A w_2$. In particular $w_2$ vanishes in $H^*\BSpin$, where also $w_2=s_2$. By definition $s_{2^k+1}$ is the sum $\sum_ix_i^{2^k+1}$ expressed in terms of the elementary symmetric polynomials $w_j$ in the $x_i$. We can compute the total square easily:
	\begin{align*}
		\Sq\left(s_{2^k+1}\right)&=\Sq\left(\sum x_i^{2^k+1}\right)\\
		&=\sum\left(x_i+x_i^2\right)^{2^k+1}\\
		&=\sum\left(\left(x_i+x_i^2\right)\left(x_i+x_i^2\right)^{2^k}\right)\\
		&=\sum\left(\left(x_i+x_i^2\right)\left(x_i^{2^k}+x_i^{2^{k+1}}\right)\right)\\
		&=\sum(\left(x_i^{2^k+1}+x_i^{2^k+2}+x_i^{2^{k+1}+1}+x_i^{2^{k+1}+2}\right)\\
		&=s_{2^k+1}+s_{2^k+2}+s_{2^{k+1}+1}+s_{2^{k+1}+2}.
	\end{align*}
	In particular $\Sq^{2^k}\left(s_{2^k+1}\right)=s_{2^{k+1}+1}$, so by induction $s_{2^k+1}$ vanishes for all $k\ge 0$ in $H^*\BSpin$.
\end{proof}

\begin{remark}\label{b5}
	A potential point of confusion arises when we identify $H^*\BSpin$ with $\Z_2[w_n:n>2,n\neq 2^k+1]$. Namely, this is an isomorphism of rings, but not of $\m A$-modules, assuming the $\m A$-action on Stiefel-Whitney classes is inherited from the action on $H^*\BO$. It is better to therefore identify $H^*\BSpin$ with the quotient of $H^*\BSO=\Z_2[w_n:n\ge 2]$ by $\m Aw_2$. To highlight the potential issue, consider the class $s_{17}$. Lemma \ref{z} shows $s_{17}$ is trivial in $H^*\BSpin$, but direct computation (using Newton's identities, for example) shows that $z_{17}$ does not vanish when simply setting $w_n=0$ for all $n\ge 2$, $n=2^k+1$. In particular $s_{17}=w_7w_{10}+w_6w_{11}+w_4w_{13}$. This is a low-degree case which demonstrates $H^*\BSpin$ is not isomorphic to $\Z_2[w_n:n>2,n\neq 2^k+1]$ as an $\m A$-module.
\end{remark}

\begin{theorem}
	We have $PH^n\BSpinc\cong\Z_2$ whenever $n\ge 2$ and $n\neq 2^k+1$, and $PH^n\BSpinc=0$ for other $n$. Let $z_n$ be the primitive generator of $PH^n\BSpinc$ in degree $n$ for some $n\ge 2$ with $n\ge 2^k+1$. Then
		\[z_n=\begin{cases}
		s_n\pw \A(n)\ge 3,\\
		s_{n/2,n/2}\pw \A(n)=2,\\
		w_2^{n/2}\pw \A(n)=1.
	\end{cases} \]
\end{theorem}
\begin{proof}
	The Girard formula shows that $s_n$ is $nw_n$ plus decomposable elements. Thus $z_n=s_n$ for all odd $n$. If $n$ is a power of $2$, then $z_n=w_2^{n/2}$. The remaining case is when $n=2^im$ for odd $m>1$ and $i>0$. When $m\neq 2^k+1$, $s_m$ is primitive, so $z_n=s_m^{2^i}=s_{2^im}=s_n$. If $m=2^k+1$, one can show that $s_{m,m}$ is primitive, and thus $z_n=s_{m,m}^{2^{i-1}}=s_{n/2,n/2}$. This proves the claim.
\end{proof}

Note that over $\Z$, $s_n^2=2s_{n,n}+s_{2n}$, so $s_{n,n}=\frac12(s_n^2-s_{2n})$. We can use this to compute $s_{n,n}$ modulo $2$ if we know $s_n$ over $\Z$.

\chapter{Computation of the $\Spinc$ transfer map}

All coefficients are taken in $\Z_2$ unless stated otherwise. In this proof we refer to integration over the fiber, which we define first here: let $X$ be a Poincar\'e duality space of formal dimension $n$ with a (known) orientation, and let $X\to E\xrightarrow{\pi} B$ be a fibration for which $\pi_1B$ acts on $X$ by orientation preserving homotopy equivalences. Then integration along the fiber $\pi_!\colon H^{k+n}E\to H^kB$ is defined as the composite
	\[H^{k+n}E\twoheadrightarrow E^{k,n}_\infty\hookrightarrow E_2^{k,n}=H^k(B;H^nX)\to H^kB. \]
The first two maps come from the fact that $E_2^{k,\ell}=0$ for $\ell>n$ and the last map comes from the orientation.

\begin{theorem}\label{thm}
	The group $G\coloneqq\SU(3)$ acts transitively on the space $\CP^2$ with fiber $H\coloneqq S(\U(1)\times\U(2))$. This gives a bundle
	\[\CP^2\to B(S(U(2)\times U(1))\xrightarrow{\pi}\BSU(3). \]
	with associated vertical bundle $\tau$. We exhibit classes cohomology $x_i,y_i$ and prove the following:
	\begin{enumerate}
		\item $H^*BH\cong\Z_2[x_2,x_4]$;
		\item $H^*BG=\Z_2[y_4,y_6]$;
		\item $B\pi^*y_4=x_2^2+x_4$ and $B\pi^*y_6=x_2x_4$;
		\item we have
			\begin{enumerate}
				\item $\Sq^1(y_4)=0$ and $\Sq^2(y_4)=y_6$;
				\item $\Sq^1(y_6)=0$ and $\Sq^2(y_6)=0$;
			\end{enumerate}
		\item $w(\tau)=1+x_2+x_4$
		\item $\pi_!\colon H^nBH\to H^{n-4}BG$ is given modulo $y_6$ by
			\[\pi_!(x_2^ax_4^b)=\begin{cases}
				y_4^{a/2-1}\pw a>0\text{ is even and }b=0,\\
				y_4^{b-1}\pw a=0\text{ and }b>0,\\
				0\pw\text{ otherwise.}
			\end{cases}\]
	\end{enumerate}
\end{theorem}

Proof of $(1)-(4)$: We have a commuting diagram as shown, where the horizontal rows are fiber bundles.
\begin{figure}[h!]\centering
	\begin{tikzcd}
		\U(1)\ar[r]\ar[d,equal] & BS(\U(1)\times\U(2))\ar[r,"f"]\ar[d,"\pi"] & B(\U(1)\times\U(2))\ar[d,"j"]\\
		\U(1)\ar[r]\ar[d,equal] & \BSU(3)\ar[r,"g"]\ar[d] & \BU(3)\ar[d]\\
		\U(1)\ar[r] & E\U(1)\ar[r] & \BU(1)
	\end{tikzcd}
\end{figure}

The map $\BU(3)\to\BU(1)$ is induced by the determinant. Comparing spectral sequences shows that the generator of $H^1\U(1)$ transgresses to $c_2$, where $H^*\BU(3)=\Z_2[c_2,c_4,c_6]$. Thus $H^*\BSU(3)=\Z_2[y_4,y_6]$ where $g^*c_4=y_4$ and $g^*c_6=y_6$. Next, write $H^*\BU(1)=\Z_2[a_2]$ and $H^*\BU(2)=\Z_2[b_2,b_4]$. Using this to identify $H^*B(\U(1)\times\U(2))=\Z_2[a_2,b_2,b_4]$, we can compute $j^*$ via the product formula for Chern classes as shown.
\begin{alignat*}{4}
	g^*c_2&=0& \hspace{1in} j^*c_2&=a_2+b_2&\\
	g^*c_4&=y_4& j^*c_4&=a_2b_2+b_4&\\
	g^*c_6&=y_6& j^*c_6&=a_2b_4.&
\end{alignat*}

By composing $f$ with the projection to $\BU(2)$, we obtain generators $x_2=f^*b_2$ and $x_4=f^*b_4$ in $H^*BS(\U(1)\times\U(2))=\Z_2[x_2,x_4]$. Comparing spectral sequences again, the generator of $H^1\U(1)$ in the top row transgresses to $j^*c_2=a_2+b_2$, and this is the only nontrivial differential. Thus we also have $f^*a_2=f^*(a_2+b_2+b_2)=f^*b_2=x_2$. Finally we compute
\[\pi^*y_4=\pi^*g^*c_4=f^*j^*c_4=f^*(a_2b_2+b_4)=x_2^2+x_4\]
and
\[\pi^*y_6=\pi^*g^*c_6=f^*j^*c_6=f^*(a_2b_4)=x_2x_4,\]

thus establishing parts $(1)-(3)$ of the theorem. For part $(4)$, note that $\Sq^1y_4=\Sq^1y_6=0$ for dimension reasons. Temporarily write $x,y,z$ for the (mod $2$) Chern roots of $\BU(3)$. Then $\Sq^2c_4=c_2c_4+c_6$ and $\Sq^2c_6=c_2c_6$ as shown.
\begin{alignat*}{4}
	\Sq^2c_4&=\Sq^2(xy+xz+yz)\hspace{2in}&\Sq^2c_6&=\Sq^2(xyz)&\\
	&=(x+y)xy+(x+z)xz+(y+z)yz& &=(x+y+z)(xyz)\\
	&=(x+y+z)(xy+xz+yz)+3xyz& &=c_2c_6\\
	&=c_2c_4+c_6&
\end{alignat*}
Thus $\Sq^2y_4=\Sq^2g^*(c_4)=g^*\Sq^2(c_4)=g^*(c_2c_4+c_6)=y_6$, and similarly $\Sq^2y_6=0$.
\\

Proof of $(5)$: Let $\rew\tau$ be the vertical bundle of $B(\U(2)\times\U(1))\to\BU(3)$. The total space of the universal $\U(3)$-bundle is $E\U(3)\times_{\U(3)}\f g$, where $\f g$ is the Lie algebra of $\U(3)$. Considering $B(\U(2)\times\U(1))$ as $E\U(3)/(\U(2)\times\U(1))$, we see that
\[\rew\tau=E\U(3)\times\f h^\perp \]
where $\f h^\perp\cong \f g/\f h$ is a chosen orthogonal complement and $\f h$ is the Lie algebra of $\U(2)\times\U(1)$.
\\

We now compute the action of $\U(2)\times\U(1)$ on $\f h^\perp$. An arbitrary element in $\U(2)$ and $\f h^\perp$ can be written respectively as
\[\begin{pmatrix}
	a & b & 0\\ -u\ol b & u\ol a & 0\\ 0 & 0 & z
\end{pmatrix}\and \begin{pmatrix}
	0 & 0 & v_1\\ 0 & 0 & v_2\\ -\ol v_1 & -\ol v_2 & 0
\end{pmatrix}, \]
where $a\ol a+b\ol b=1$ and $u\ol u=z\ol z=1$. We write the former as $(P,z)$ and latter as $v$. The adjoint action is then $(P,z)v(P,z)\inv$. As matrices this takes $v$ to
\[\begin{pmatrix}
	0 & 0 & \ol z(av_1+bv_2)\\
	0 & 0 & u\ol z(\ol av_2-\ol bv_1)\\
	-z(\ol a\ol v_1+\ol b\ol v_2) & \ol uz(b\ol v_1-a\ol v_2) & 0
\end{pmatrix}. \]
At the same time
\[\begin{pmatrix}
	a & b \\ -u \ol b & u\ol a
\end{pmatrix}\begin{pmatrix}
	v_1\\v_2
\end{pmatrix}=\begin{pmatrix}
	av_1+bv_2\\ u(\ol av_2-\ol bv_1)
\end{pmatrix}, \]
and thus the adjoint action of $(P,z)$ on $v$ is $v\mapsto (P,z)\cdot v=Pv\ol z$. Include $\Z_2^3\to \U(2)\times\U(1)$ with the first factor mapping to $z$ and the second two mapping to $\U(2)$ (to be consistent with earlier notation). We thus compute
\[\begin{pmatrix}
	(-1)^j & 0 & 0\\ 0 & (-1)^k & 0\\ 0 & 0 & (-1)^i
\end{pmatrix}\cdot v=\begin{pmatrix}
	(-1)^j & 0 \\ 0 & (-1)^k
\end{pmatrix}\begin{pmatrix}
	v_1 \\ v_2
\end{pmatrix}(-1)^i=\begin{pmatrix}
	(-1)^{i+j}v_1\\ (-1)^{i+k}v_2
\end{pmatrix}. \]
Writing $H^*B\Z_2^3=\Z_2[x,y,z]$, we now compute
\[w(\rew\tau)=(1+x+y)(1+x+z)=1+y+z+xy+xz+yz+x^2.\]
Previously we wrote $H^*B(\U(1)\times\U(2))=\Z_2[a_2,b_2,b_4]$ and identified $a_2=x$, $b_2=y+z$, and $b_4=yz$. This shows $w(\rew t)=1+b_2+a_2b_2+b_4+a_2^2$. Finally, we pull back to $BS(\U(1)\times\U(2))$ to get $w(\tau)=1+x_2+x_2^2+x_4+x_2^2=1+x_2+x_4$.
\\

Proof of $(6)$: Previously we saw $H^*BH=\Z_2[x_2,x_4]$ and $H^*BG=\Z_2[y_4,y_6]$ where $\pi^*y_4=x_2^2+x_4$ and $\pi^*y_6=x_2x_4$. We now set $H^*\CP^2=\Z_2[u]/(u^3)$. For degree reasons, the spectral sequence associated to $\CP^2\to BH\to BG$ is trivial, so we necessarily have $i^*x_2=u$. Then $i^*x_2^2=u^2$, and since $\pi\circ i$ is trivial,
\[i^*x_4=i^*(x_4+x_2^2+x_2^2)=i^*(\pi^*(y_4)+x_2^2)=u^2.\]

Now as a module over $H^*BG$, $H^*BH$ is free with basis $\{1,u,u^2\}$. Thus for any $x\in H^*BH$ we can uniquely identify $x$ with $r_0(x)+r_1(x)u+r_2(x)u^2$, where $r_i(x)\in H^*BG$. Integration along the fiber of $\tau$ is an $H^*BG$-module morphism $H^nBH\to H^{n-4}BG$, so it remains to determine this map on the basis elements $1,u,u^2$. For degree reasons $1$ and $u$ map to $0$, and $u^2$ maps to either $1$ or $0$. Since $u^2$ is already an element on $H^*\CP^2$, it is in the image of $H^4BH\to E_\infty^{0,4}\to E^{0,4}_2$, so $u^2$ maps to $1$.
\\

To compute the transfer map we now must write all monomials in $x_2,x_4$ in terms of $x_2^2+x_4=\pi^*(y_4)$ and $x_2x_4=\pi^*(y_6)$ in the basis $1,x_2,x_2^2$. Note that $\pi_!$ is a $H^*BG$-module map, so for $y\in H^*BG$ and $x\in H^*BH$, we have $\pi_!(\pi^*(y)x)=y\pi_!(x)$. Since
\[\pi_!(x_2^{n+k}x_4^n)=\pi_!((x_2x_4)^nx_2^k)=\pi_!(\pi^*(y_6^n)x_2^k)=y_6^n\pi_!(x_2^k)\and \pi_!(x_2^nx_4^{n+k})=y_6^n\pi_!(x_4^k),\]
we only need to compute $\pi_!(x_2^n)$ and $\pi_!(x_4^n)$. We also observe for $n\ge 3$
\[\pi_!(x_2^n)=\pi_!(x_2^{n-2}(x_2^2+x_4)+x_2^{n-2}x_4)=\pi_!(x_2^{n-2})y_4+\pi_!(x_2^{n-3})y_6\]
and
\[\pi_!(x_4^n)=\pi_!(x_4^{n-1}(x_4+x_2^2)+x_4^{n-1}x_2^2)=\pi_!(x_4^{n-1})y_4+\pi_!(x_4^{n-3})y_6^2. \]
It now remains to compute the six base cases $\pi_!(x_2^i)$ and $\pi_!(x_4^i)$ for $0\le i\le 2$. For degree reasons $\pi_!(1)=\pi_!(x_2)=0$, and our chosen Leray-Hirsch isomorphism gives $\pi_!(x_2^2)=1$. Then $\pi_!(x_4)=\pi_!(x_2^2+x_4)+\pi_!(x_2^2)=y_4\pi_!(1)+1=1$, and similarly
\[\pi_!(x_4^2)=\pi_!((x_2^2+x_4)^2)+\pi_!(x_2^4)=\pi_!(x_2^2(x_2^2+x_4))+\pi_!(x_2^2x_4)=\pi_!(x_2^2)y_4+\pi_!(x_2)y_6=y_4. \]
From the recurrence relation, it follows that, modulo $y_6$, we have $\pi_!(x_2^{2n})=y_4^{n-1}$, $\pi_!(x_2^{2n+1})=0$, and $\pi_!(x_4^n)=y_4^{n-1}$ for $n>0$. This implies statement (6) and completes the proof of theorem \ref{thm}.

\end{appendicesmulti} 


\begin{SingleSpace} 
\printbibliography[title={REFERENCES CITED}, category=cited]
%
\end{SingleSpace}

\end{document}